\documentclass[12pt,a4paper]{article}
\usepackage[left=1.5cm,right=2.4cm,top=2cm,bottom=2cm]{geometry}

\usepackage{verbatim}

\usepackage{array}
\usepackage{delarray}
\usepackage[utf8]{inputenc}
\usepackage[colorlinks,allcolors=blue]{hyperref}
\usepackage{bigints}

\usepackage{multirow}  
\usepackage{tabularx}
\usepackage{graphicx} 

\usepackage{multirow}
\usepackage{amsmath} 
\usepackage{amsthm}
\newtheorem{corollary}{Corollary}
\newtheorem{theorem}{Theorem}[section]

\newtheorem{remark}{Remark}
\newtheorem{Solution}{Solution}[section]

\newtheorem{example}{Example}[section]
\newtheorem{Note}{Note}
\newtheorem{lemma}{Lemma}[section]
\usepackage{amsmath, amsfonts, amssymb}
\numberwithin{equation}{section}
\newcommand{\cmmnt}[1]{\ignorespaces}

\providecommand{\keywords}[1]
{
  \small	
  \textbf{\textit{Keywords---}} #1
}

\author{}
\date{January 2023}

\begin{document}
\Large
\begin{center}
\textbf{Estimation of Tsallis entropy for exponentially distributed several populations} 

\hspace{5pt}

\large
Naveen Kumar$^1$, Ambesh Dixit$^2$, Vivek Vijay$^{3}$ \\

\hspace{2pt}

\small  
$^{1)}$ Department of Mathematics, IIT Jodhpur, Rajasthan, India\\
\href{mailto:kumar.248@iitj.ac.in }{kumar.248@iitj.ac.in} \\
$^{2)}$ Department of Physics, IIT Jodhpur, Rajasthan, India\\
\href{mailto:vivek@iitj.ac.in }{ambesh@iitj.ac.in} \\
$^{3)}$ Department of Mathematics, IIT Jodhpur, Rajasthan, India\\
\href{mailto:ambesh@iitj.ac.in }
{vivek@iitj.ac.in}\\


\end{center}
\begin{abstract}
We study the estimation of Tsallis entropy of a finite number of independent populations, each following an exponential distribution with the same scale parameter and distinct location parameters for $q>0$. We derive a Stein-type improved estimate, establishing the inadmissibility of the best affine equivariant estimate of the parameter function. A class of smooth estimates utilizing the Brewster technique is obtained, resulting in a significant improvement in the risk value. We computed the Brewster-Zidek estimates for both one and two populations, to illustrate the comparison with best affine equivariant and Stein-type estimates. We further derive that the Bayesian estimate, employing an inverse gamma prior, which takes the best affine equivariant estimate as a particular case. We provide a numerical illustration utilizing simulated samples for a single population. The purpose is to demonstrate the impact of sample size, location parameter, and entropic index on the estimates.

\end{abstract}


\section{Introduction}
The concepts of entropy as a degree \cmmnt{of disorder in a physical system}of disorder within a physical system was first introduced in the $19$th century by Clausius, Boltzmann, and Gibbs in the domains of thermodynamics and statistical mechanics. Shannon\cite{shannon1948mathematical} significantly enhanced the concept by linking it to the communication theory, considering it as a information indicator. The use of Shannon entropy in data analysis has been the subject of a substantial body of research, which has been published extensively. A collection of such methodologies and their applicability for the risk management, portfolio selection, categorical data analysis and econometrics can be seen in {\cite{eshima2020statistical_bookOnCategorical_data_analysis,maasoumi1993compendium_Econometric_methods_based_on_entropy2,golan2008information_Econometric_methods_based_on_entropy1,philippatos1972entropy_marketrisk_and_selecting_efficientportfolio,pola2016entropy_portfolio}}. Tsallis entropy{\cite{tsallis1988possible}} $S_q$, is a generalization of Shannon entropy and \cmmnt{for a continuous random variable $X$}considering a random variable $X$ that is continuous, it is computed as 
\begin{equation}
    S_{q}(X)=\frac{1}{q-1}\left(1-\int_{-\infty}^{\infty}p^{q}(x)dx \right),
\end{equation}
where $p(x)$ is the probability density function of $X$ and $q$ is a non-unit real number. Recently, Tsallis entropy has been widely used in data analysis, with applications in areas, like financial markets{\cite{gurdgiev2016financialmarket,wang2018financialmarket}}, image thresholding{\cite{bhandari2015imagethresholding,de2004imagethresholding,sarkar2013imagethresholding}}, pattern recognition{\cite{ribeiro2017patternrecognition,zhang2008patternrecognition}}, and signal processing{\cite{beenamol2012waveletsignalprocessing,singh2020signalprocessing}} etc.

\cmmnt{Over the last two decades, considerable amount of work has been}Over the past two decades, a substantial amount of work has been devoted to estimating the entropy function. Several commonly employed non-parametric estimators of Shannon entropy include the precise local expansion based estimator{\cite{paninski2003estimationShannonEntropyNonparametricestimation1}}, the weighted affine combination of estimators{\cite{sricharan2013ensembleestimationShannonEntropyNonparametricestimation2}}, and the Bayes estimator{\cite{holste1998bayesBayesEstimatorfortsallis}} of entropy. One noteworthy observation of estimation is that parametric estimation \cmmnt{exhibits superior performance compared to other methods, particularly in terms of}outperforms other methods in terms of risk, when applied to well-know population scenarios. For a number of multivariate distributions, including normal, exponential, and logistic, {\cite{ahmed1989entropy}} provides entropy expressions and parametric estimates.
Maya et al.{\cite{gupta2010parametricentropyestimate}} provide the Bayesian estimates of Shannon entropy for uniform, Gaussian, Wishart, inverse Wishart distributions and the performance is increased in comparison to both maximum likelihood and nonparametric estimators. The maximum likelihood method is employed to compute the estimate{\cite{shrahili2022estimation}} of entropy metric for an Log-Logistic distribution using progressive type II censoring. The estimation of entropy for generalized exponentially distributed record values, as derived in reference{\cite{chacko2018estimation}}, involves the utilization of both maximum likelihood and Bayesian approaches. A comprehensive inference procedure is outlined, accompanied by an illustrative application to a real dataset. \cmmnt{Kang et al.\cite{kang2012estimation} derived the estimators for the entropy function of the double exponential distribution using both maximum likelihood estimation and approximation maximum likelihood estimation.}Kang et al.{\cite{kang2012estimation}} obtained the entropy estimators for the double exponential distribution through both maximum likelihood estimation and approximate maximum likelihood estimation. The derivation is based on multiply Type-II censored data, and the study compares the parametric and non-parametric estimates.

In $1974$, Brewster et al.{\cite{brewster1974improving}} introduced a technique, inspired by Stein's methodology, to improve the best affine equivariant estimator(BAEE) under the bowl-shaped error functions. Mishra et al.{\cite{misra2005estimation}} employ the Brewster technique to improve the BAEE for the entropy of multivariate normal distribution and prove that BAEE is inadmissible. Also, the Brewster-type estimator under the quadratic loss function is generalised Bayes. Kayal et al.{\cite{kayal2011estimating}} derived \cmmnt{the BAEE for the entropy of both scale parameter exponential and location-scale parameter exponential distributions under the linex loss function.}the BAEE for both scale parameter exponential and location-scale parameter exponential population entropy value under the linear exponential loss function. Additionally, \cmmnt{the study presents the derivation of Stein and Brewster-type improved estimates over the BAEE.}the study shows the dominance of Stein and Brewster-type estimates improving the BAEE. The estimation problem of Renyi entropy for multiple shifted exponential distributions having equal location parameter and unequal scale parameters is addressed in reference {\cite{kayal2015estimating}}. \cmmnt{They derived the uniformly minimum variance unbiased estimator, BAEE, and for a single population, computed an estimate that dominated BAEE under quadratic loss function, using the Brewster technique.}Also, The derivation of UMVUE, BAEE, and for one population is presented. It provides an estimate that dominated BAEE under quadratic loss function, using the Brewster technique. The identical problem of estimating Renyi entropy for multiple exponential populations with varying scale parameters under the linex loss function is investigated in reference {\cite{kayal2017estimating}}. Petropoulos et al.{\cite{petropoulos2020improved}} computed the equivariant estimate for the Shannon entropy of the mixture model with the square error loss function for the exponential distribution. The study further established that the generalized Bayes estimator coincides with the Brewster estimator. Recently, BAEE{\cite{patra2020estimating}} of the function $\ln{\sigma}$, of scale parameter of exponential distribution \cmmnt{with unknown location parameter} under an arbitrary location invariant bowl-shape loss function is shown to be inadmissible and the Kubokawa approach is also discussed.

Our focus is on computing the equivariant estimation of Tsallis entropy for multiple populations following exponential distributions. \cmmnt{To the best of our knowledge,}To our knowledge, this specific problem has not yet been addressed in the literature so far. The problem of estimating Tsallis entropy has not been well investigated from various perspectives. A notable contribution to non-parametric estimation of tsallis entropy includes the application of the sample spacing technique as employed in {\cite{wachowiak2005estimation}}, the development of the balanced estimator{\cite{Bonachela_2008NONparametricEstimation3}} that strikes a compromise between low bias and small statistical errors, and the formulation of the plug-in estimator{\cite{NONparametricEstimation1}} for \cmmnt{second-order} the Tsallis entropy with index value $q=2$, based on the kernel density estimator. Presently, there has been an increasing interest among researchers in the parametric estimation of Tsallis entropy. Some recent work includes Tsallis entropy estimation of inverse Lomax distribution for multiple censored data{\cite{bantan2020estimation}}, for kumaraswamy distribution using beta function{\cite{al2021estimation}}, for the power function distribution in the presence of outliers{\cite{hassan2022estimation}} and for Log-Logistic distribution with progressive type II censoring{\cite{shrahili2022estimation}}.

In this paper, the BAEE is derived, for the Tsallis entropy of exponentially distributed populations, each with the same scale parameter for $q>0$, under the strictly bowl-shaped quadratic loss function. We study the Stein-type estimator, which provides a notable improvement over BAEE for suitable sample values. The Stein-type estimator exhibits non-smooth behaviour with respect to the estimated parameter. As discussed in {\cite{brewster1974improving}}, a class of improved smooth estimators is derived to improve the performance of the BAEE based on the quadratic loss function. This study is inspired from the recent works on the improvement of the BAEE for various applications in literature indicated above.

The following segments of the paper are arranged as follows: In Section \ref{The Best Affine Equivariant Estimator section1}, we derive the BAEE of $1/\sigma^{k(q-1)},$ for $q>0$, under the quadratic loss function. A class of improved Stein-type estimate and its existence is presented in Section \ref{Stein-type improved estimate section}. In Section \ref{Smooth improved estimate section4}, we apply the orbit-wise risk reduction Brewster-Zidek technique to obtain the improved smooth estimate over BAEE. The Bayesian estimate is derived by considering the prior, inverse gamma distribution, of the scale parameter. A study based on the simulation is provided to illustrate the estimators performance in Section \ref{numerical comparisons section6}. Finally, the conclusion of the paper is in Section \ref{conclusion section7}.


\section{BAEE for the Tsallis entropy of multiple exponential populations}\label{The Best Affine Equivariant Estimator section1}
In this section, we derive the Tsallis entropy for $k(\ge1)$ independent populations $P_i(i=1,2,...,k)$, where each population is distributed with a probability density function,

\begin{equation} \label{twoparaExponentialdistr}
p_{i}(x;u_i,\sigma)=
\begin{cases}
\frac{1}{\sigma}e^{-\frac{x-u_i} {\sigma}}, & \text{if $x>u_i$,} \\
0, & \text{otherwise.}
\end{cases}
\end{equation}
Let us consider a random sample $\{x_i^1, x_i^2,..., x_i^n\}$ drawn from the population $P_i$, it can be observed that the Tsallis entropy for the corresponding distribution exists finitely if and only if $q>0$, and given by 
\begin{equation} \label{tsallisentropyforexponentialdistr}
    S_q(P_i)=\frac{1}{q-1}\left( 1-\frac{1}{q{\sigma}^{q-1}}\right).
\end{equation}
It is well known that the Tsallis entropy exhibits $q$-additivity for independent distributions and that the measure of randomness is unaffected by the location parameter, our first step is to determine the Tsallis entropy for $k$ number of independent exponential distributions with same scale and distinct location parameter.
\begin{lemma}
    Given $k$ independent exponential distributions as in equation(\ref{twoparaExponentialdistr}), the joint Tsallis entropy is $S_q(P_1,P_2,...,P_k)=\frac{1}{q-1}\left(1-{\left[1+(1-q)S_q(P_i)\right]}^k\right)$.
\end{lemma}
\begin{proof}
Let $\Delta=\frac{1}{q{\sigma}^{q-1}}$, then for equation(\ref{tsallisentropyforexponentialdistr}), we can write
\begin{align*}
 S_q(P_1)&=\frac{1}{q-1}(1-\Delta). 
 \end{align*}
From the $q-$additivity of Tsallis entropy, we get
 \begin{align*}
 S_q(P_1,P_2)&=\frac{1}{q-1}(1-\Delta)\left(\Delta+1 \right),\\
 S_q(P_1,P_2,P_3)&=\frac{1}{q-1}(1-\Delta)\left(\Delta^2+\Delta+1 \right), \\
 &\vdots\\
    S_q(P_1,P_2,...,P_k)&=\frac{1}{q-1}(1-\Delta)(\Delta^{k-1}+\Delta^{k-2}+\cdots+\Delta+1)\\
    &=\frac{1}{q-1}(1-\Delta^k)\\
    &=\frac{1}{q-1}\left[1-{\left(\frac{1}{{q\sigma^{q-1}}}\right)}^{k} \right]=\frac{1}{q-1}-\frac{1}{(q-1){q^k}{\sigma}^{k(q-1)}}\\
    &=\frac{1}{q-1}\left(1-{\left[1+(1-q)S_q(P_i)\right]}^k\right).
\end{align*} 
\end{proof}

\noindent This motivate us to consider the equivariant estimation problem for the function $\Theta(\sigma)=1/\sigma^{k(q-1)}$ for a fixed value of $k$ and $q$, in order to estimate tsallis entropy for a given sample, under the quadratic error function $L(t)={(t-1)}^2$, which is a strictly bowl-shaped function with minimum at $t=1$, can be expressed as,
\begin{equation} \label{lossfunction}
    L\left(\frac{\delta(X)}{\Theta(\sigma)}\right)={\left(\left(\frac{\delta(X)}{\Theta(\sigma)}\right)-1 \right)}^2.
\end{equation}

\noindent Based on the $i$-th random sample $\{x_i^1, x_i^2,..., x_i^n\}$, for $({\mu}_{i},\sigma)$, $(X_{i}^{(1)},Y_{i})$ forms a complete and sufficient statistics, where, $X_{i}^{(1)}=min_{1\le j \le n}X_{i}^{j}$ and $Y_{i}=\sum_{j=1}^{n}X_{i}^{j}$. $X_{i}^{(1)}$ and $Z_{i}$ are distributed independently, where $Z_{i}=Y_{i}-nX_{i}^{(1)}$. Also $X_{i}^{(1)}$ is exponentially distributed with parameters (${\mu}_i $,$\sigma/n$) and $2Z_{i}/{\sigma}$ follows $\chi_{2(n-1)}^{2}$ (see page no. 43 of {\cite{lehmann}}). Define $T=\sum_{i=1}^{k}\sum_{j=1}^{n}\left(X_i^{j}-X_{i}^{(1)} \right)$ then for $(\Bar{\mu},\sigma)$, $(X^{(1)},T)$ forms a complete and sufficient statistics where $\Bar{\mu}=(\mu_1,\cdots,\mu_k)$. Also $X^{(1)}$ and $T$ are independently distributed. Utilizing the additivity characteristic of the chi-square distribution, we can infer that $2T/\sigma$ follows a chi-square distribution with $2k(n-1)$ degrees of freedom. The plug-in estimator of $\Theta(\sigma)$, employing the maximum likelihood estimator(MLE) of $\sigma$, is expressed as ${(kn)}^{k(q-1)}/{T}^{k(q-1)}$.

Consider the affine transformations, $h_{r,s_i}(x_{ij})=rx_{i}^{j}+s_i,$ where $r>0,s_i\in \mathbb{R},$ $ j=1,...,n$, and $i=1,...,k$. Let $\Bar{s}=(s_1,...,s_k)$ and $g_{r,\Bar{s}}=(g_{r,s_1},...,g_{r,s_k})$, under the transformations $g_{r,\Bar{s}}$, we observe that, 
\begin{equation}\label{equivariant condition}
  (\Bar{\mu}, \sigma)\rightarrow (r\Bar{\mu}+\Bar{s},r\sigma), \ \ \ \ (X^{(1)},T)\rightarrow (rX^{(1)}+\Bar{s},rT).  
\end{equation}
\noindent Accordingly, we get $\Theta(\sigma)=1/{\sigma}^{k(q-1)}\rightarrow 1/r^{k(q-1)}{\sigma}^{k(q-1)}$. The loss function (\ref{lossfunction}), if ${\delta} \rightarrow {\delta}/r^{k(q-1)}$, is invariant under the group $G_{r,\Bar{s}}$ of affine transformation. The obtained affine equivariant estimator has the 
general form
\begin{equation}\label{BAEEgeneral form}
    {\delta}_{c}(X^{(1)},T)=\frac{c}{T^{k(q-1)}},
\end{equation}
\noindent where $c$ is any constant. The density functions of random variable $T$, needed to calculate the expectations in subsequent results is given by
\begin{equation}
    f_{T}(t)=\frac{t^{k(n-1)-1} e^{-\frac{t}{\sigma}}}{{\sigma}^{k(n-1)}{\Gamma{(k(n-1))}}}, \text{ \ \ \ $t>0$,} 
\end{equation}

\noindent In the next theorem, the BAEE
of $\Theta(\sigma)$ is derived.
\begin{theorem}\label{BAEE}
    For the quadratic error function(\ref{lossfunction}), the BAEE of $\Theta(\sigma)$ exists for $q<(n+1)/2$, and is ${\delta}_{c_0}(X^{(1)},T)$, where $c_{0}=\frac{\Gamma(k(n-1)+k(1-q))}{\Gamma(k(n-1)+2k(1-q))}$.
\end{theorem}
\begin{proof}
    The risk value of the estimator is given by
    \begin{equation*}
        R(\sigma,\delta)=E_{\overline{\mu},\sigma}\left(L\left(\frac{\delta}{\Theta(\sigma)}\right)\right).
    \end{equation*}
    Using the general form of equivariant estimator and from equivariant property, simplification gives
    \begin{equation}\label{lossfunctionvalue}
    \begin{split}
        R(\sigma,\delta)&=E_{\overline{0},1}{\left(c{T}^{k(1-q)}-1 \right)}^{2} \\
        &= {c^2}E_{\overline{0},1}(T^{2k(1-q)})-{2c}E_{\overline{0},1}(T^{k(1-q)})+1.
        \end{split}
    \end{equation}
    Here $\overline{0}=(0,0,...,0)$ is used. Differentiating right side with respect to $c$, gives isolated point
    \begin{equation*}
    \begin{split}
        c_0&=\frac{E_{\overline{0},1}{(T}^{k(1-q)})}{E_{\overline{0},1}{(T}^{2k(1-q)})}\\
        &=\frac{\Gamma(k(n-1)+k(1-q))}{\Gamma(k(n-1)+2k(1-q))}.
         \end{split}
         \end{equation*}
We utilize the moments of $\frac{2T}{\sigma}\sim {\chi_{(2k(n-1))}^{2}}$. This completes the proof. 
\end{proof}
\noindent The risk function of the equivariant estimator is equal for all parameter values. The risk value of minimum risk equivariant estimator ${\delta}_{c_0}(X^{(1)},T)$, is computed in the following corollary
\begin{corollary}
    For the quadratic error function (\ref{lossfunction}), the BAEE ${\delta}_{c_0}(X^{(1)},T)$ has the risk value $1-\frac{{\left(\Gamma(k(n-1)+k(1-q))\right)}^2}{\Gamma(k(n-1)+2k(1-q))\Gamma(k(n-1))}$.
\end{corollary}
\begin{proof}
    As $2T/{\sigma}\sim {\chi}_{2k(n-1)}^2$ and for a chi-square distributed random variable $X$ with degrees of freedom $d$, the $m$-th moments, where $m>-\frac{k}{2}$, can be expressed as $E(X^m)=\frac{2^m\Gamma(\frac{d}{2}+m)}{\Gamma(\frac{d}{2})}$. So from (\ref{lossfunctionvalue}), we get
    \begin{equation*}
    \begin{split}
        R(\sigma,\delta)&=1-\frac{{\left(E_{0,1}{(T}^{k(1-q)})\right)}^{2}}{E_{0,1}{(T}^{2k(1-q)})}\\
        &=1-\frac{{\left(\Gamma(k(n-1)+k(1-q))\right)}^2}{\Gamma(k(n-1)+2k(1-q))\Gamma(k(n-1))}.
        \end{split}
    \end{equation*}
\end{proof}
\noindent In the next corollary, the confidence level of the equivariant estimate is derived.
\begin{corollary}
    For a given level of confidence $1-{\alpha}(0<\alpha<1)$, the confidence interval for $\Theta(\sigma)$, is $\left({\pi}/{({\chi}_{k(n-1),1-\frac{\alpha}{2}} )}^{k(1-q)},{\pi}/{( {{\chi}_{k(n-1),\frac{\alpha}{2}}})}^{k(1-q)}\right)$, where ${\pi}=(2T)^{k(1-q)}$. 
\end{corollary}
\begin{proof}
    Let $\pi_i={\pi}/c_i, \ \ i=1,2$, then for $0<\alpha<1$, consider
    \begin{equation*}
    \begin{split}
        P\left(\pi_2<\sigma^{k(1-q)}<\pi_1\right)&\ge 1-\alpha \\
     \implies   P\left(c_1<{\left(\frac{2T}{\sigma} \right)}^{k(1-q)}<c_2\right)&\ge1-{\alpha} 
    \end{split}
    \end{equation*}
    As $2T/\sigma$ takes non-negative values only so using the monotonic behaviour of ${\left(\frac{2T}{\sigma}\right)}^{1-q}$, we have
    \begin{equation*}
     P\left({c_1}^{{\frac{1}{k(1-q)}}}<\frac{2T}{\sigma}<{{c_2}^{\frac{1}{k(1-q)}}}\right)\ge1-{\alpha}    
    \end{equation*}
    From the distribution of $2T/\sigma$, we get ${c_1}^{\frac{1}{k(1-q)}}={{\chi}_{k(n-1),\frac{\alpha}{2}}}$ and ${c_2}^{\frac{1}{k(1-q)}}={{\chi}_{k(n-1),1-\frac{\alpha}{2}}}$. This completes the proof.
\end{proof}
\noindent Now, we show the inadmissibility of BAEE by showing the existence of an improved estimate. For this, the Stein-type estimate is derived in next section.

\section{A Stein-type Improved Equivariant Estimate}\label{Stein-type improved estimate section}
 In order to improve the equivariant estimate $\delta_{c_0}$, we work on a larger class of estimators. Let $H_{r}=\{h_{r}: h_{r}(x)=rx,r>0\}$, be the subgroup of scaled function from group $G_{r,\Bar{s}}$. Due to this transformation, we have
 \begin{equation}\label{stein-type-invariant-transformations}
  (\Bar{\mu}, \sigma)\rightarrow (r\Bar{\mu},r\sigma), \text{\ \ \ \ } (X^{(1)},T)\rightarrow (rX^{(1)},rT). 
\end{equation}
It is observed that the error function in equation(\ref{lossfunction}) remains invariant under the scale group $H_r$, if ${\delta} \rightarrow {\delta}/r^{k(q-1)}$. Thus, the general form of scale equivariant estimate is,
\begin{equation} \label{formofOnlyLocationInvariantEstimate}
    {\delta}_{\psi}(\overline{W},T)=\frac{\psi(\overline{W})}{{T^{k(q-1)}}},
\end{equation}
where $W_i=X_{i}^{(1)}/T$, $\overline{W}=(W_1,W_2,...,W_k)$ and $\psi$ is a measurable function that takes real values. 

\noindent Assuming that all $u_{i}'$s are known, estimator $S=\sum_{i=1}^{k}\sum_{j=1}^{n}(X_{i}^{j}-{\mu}_i)$ can be used in place of estimator $T$ and it is evident that $S$ follows gamma distribution with shape and rate parameters $kn$ and $\frac{1}{\sigma}$. The non-randomised estimator of $\Theta(\sigma)$ using S, derived by following the same steps of Theorem(\ref{BAEE}), is given by
\begin{equation}
    {\delta}_{c_1}(X^{(1)},S)=\frac{c_1}{ S^{k(q-1)}},
\end{equation}
with $c_{1}=\frac{\Gamma(kn+k(1-q))}{\Gamma(kn+2k(1-q))}$.

\begin{lemma}\label{lemma of c1 and c0 inequalities}
    If $k$ is a positive integers then $\frac{\Gamma(x-k)}{\Gamma(x)}$ is a decreasing function for $x>k$.
\end{lemma}
\begin{proof}
    For $x>k$, we differentiate the function $\frac{\Gamma(x-k)}{\Gamma(x)}$, we get
    \begin{equation*}
        \begin{split}
                {\left(\frac{\Gamma(x-k)} {\Gamma(x)}\right)}^{'}&=\frac{\Gamma(x)\Gamma(x-k)\Psi(x-k)-\Gamma(x-k)\Gamma(x)\Psi(x)}{{(\Gamma(x))}^2} \\
                    &=\frac{\Gamma(x-k)}{\Gamma(x)}\left( \Psi(x-k)-\Psi(x) \right) <0 
        \end{split}
    \end{equation*}
    Above last inequality is from that digamma function $\Psi$ is a increasing function for positive real numbers. This proves that $\frac{\Gamma(x-k)}{\Gamma(x)}$ is a decreasing function for $x>k$. 
\end{proof}
\begin{Note}
        Taking $a=kn+k(1-q)$ and $b=kn+2k(1-q)$, it is clear that for $0<q<1$, we have $a<b$ and from lemma(\ref{lemma of c1 and c0 inequalities}),
    \begin{equation*}
        \begin{split}
          \frac{\Gamma(b-k)}{\Gamma(b)}&<\frac{\Gamma(a-k)}{\Gamma(a)}  \\
          \frac{\Gamma(kn+2k(1-q)-k)}{\Gamma(kn+2k(1-q))}&<\frac{\Gamma(kn+k(1-q)-k)}{\Gamma(kn+k(1-q))} \\
          \frac{\Gamma(kn+k(1-q))}{\Gamma(kn+2k(1-q))}&<\frac{\Gamma(k(n-1)+k(1-q))}{\Gamma(k(n-1)+2k(1-q))} \\
          c_1<c_0.
        \end{split}
    \end{equation*}
\end{Note}
\begin{Note}
    Taking $a=kn+k(1-q)$ and $b=kn+2k(1-q)$, it is clear that for $1<q<\frac{(n+1)}{2}$, we have $a>b$ and from lemma(\ref{lemma of c1 and c0 inequalities}),
    \begin{equation*}
        \begin{split}
          \frac{\Gamma(b-k)}{\Gamma(b)}&>\frac{\Gamma(a-k)}{\Gamma(a)}  \\
          \frac{\Gamma(kn+2k(1-q)-k)}{\Gamma(kn+2k(1-q))}&>\frac{\Gamma(kn+k(1-q)-k)}{\Gamma(kn+k(1-q))} \\
          \frac{\Gamma(kn+k(1-q))}{\Gamma(kn+2k(1-q))}&>\frac{\Gamma(k(n-1)+k(1-q))}{\Gamma(k(n-1)+2k(1-q))} \\
          c_1>c_0.
        \end{split}
    \end{equation*}
\end{Note}

\noindent The joint distribution of $T$ and $\overline{W}$ is expressed as
\begin{equation}
    f(t,\overline{w};\overline{\mu})=\frac{n^k t^{nk-1}}{{\sigma}^{nk}\Gamma({k(n-1)})}e^{\frac{n}{\sigma}\sum_{i=1}^{k}\mu_i}e^{-{\frac{t}{\sigma}\left(n\sum_{i=1}^{k}w_{i}+1 \right)}},  \text{ \ \ ${u_i/t}\le w_i <\infty$, $i=1,2,...,k$ and $0<t<\infty,$}
\end{equation}
and let $\eta=\max_{j=1}^{k}\left(u_j/w_j \right),$ then the marginal density of $\overline{W}$ is given by
\begin{equation}
    f_{\overline{W}}(\overline{w};\overline{\mu})=\frac{n^k \Gamma{(kn,\eta)}}{\Gamma({k(n-1)})}{\left(n\sum_{i=1}^{k}w_{i}+1 \right)}^{-kn}e^{\frac{n}{\sigma}\sum_{i=1}^{k}\mu_i}, \text{ \ \ \ $0\le w_i<\infty$, $i=1,2,...,k.$ }
\end{equation}
\noindent Now we show the existence of the Stein-type{\cite{stein1964inadmissibility_Estimator}} estimator in the next theorem.
\begin{theorem}\label{stein-type estimate theorem}
    Let $c_1$ be the unique solution obtained from equation 
    \begin{equation} \label{BAEEusingSderivativequation1}
        E\biggl[L^{'}\biggl(\frac{c_1}{S^{k(q-1)}} \biggr)\frac{1}{S^{k(q-1)}} \biggr]=0,
    \end{equation}
    where $S=\sum_{i=1}^{k}\sum_{j=1}^{n}(X_{i}^{j}-{\mu}_i)$. Then, for the scale equivariant error function $L(t)$, $\delta_{{\psi}^{*}}$ has nowhere risk greater than $\delta_{c_0},$ provided
    \begin{equation}\label{c1 is less than c0 in theorem1}
        E\biggl[L^{'}\biggl(\frac{c_1}{T^{k(q-1)}} \biggr)\frac{1}{T^{k(q-1)}} \biggr]<0, \text{ \ \ if $0<q<1,$}
    \end{equation}
    \begin{equation}\label{c1 is less than c0 in theorem1.2}
      \text{\ \ and \ \ }  E\biggl[L^{'}\biggl(\frac{c_1}{T^{k(q-1)}} \biggr)\frac{1}{T^{k(q-1)}} \biggr]>0, \text{ \ \ if $1<q<\frac{(n+1)}{2}$, \ \ \ \ \ \ \ \ }
    \end{equation}
    where 
    \begin{equation}
        \psi^{*}(\overline{w})=\begin{cases}
    \min\Bigl\{ {c_{0}},{c_1}{{(1+n\sum_{i=1}^{k}w_i)}^{k(1-q)}} \Bigr\} ,& \text{if \ \ $0<q<1$ and \ \ } w_{i}>0, \text{\ \ }i=1,2,...,k\\
    \max\Bigl\{ {c_{0}},{c_1}{{(1+n\sum_{i=1}^{k}w_i)}^{k(1-q)}} \Bigr\} ,& \text{if \ \ $1<q<(n+1)/2$ and \ \ } w_{i}>0, \text{\ \ }i=1,2,...,k\\
    
    c_0 ,& \text{otherwise}.
\end{cases}
    \end{equation}
    
\end{theorem}
\begin{proof}
    As the risk value of ${\delta}_{\psi},$ is dependent on $\overline{\mu}$ and $\sigma$ via the ratio $u_i/{\sigma}$ for all $i=1,2,...,k$, we can take $\sigma=1$ from equation(\ref{stein-type-invariant-transformations}). Consider the risk function of ${\delta}_{\psi},$
\begin{equation}
\begin{split}
    R(\overline{\mu},{\delta}_{\psi})
    &=E_{\overline{\mu}}\biggl[L\biggl(\frac{\psi(\overline{W})}{{T^{k(q-1)}}}\biggr)\biggr], \\
    &=E^{\overline{W}}E_{\overline{\mu}}\biggl[L\biggl({\frac{\psi(\overline{w})}{T^{k(q-1)}}}\biggr)\bigg|\overline{W}=\overline{w} \biggr].
\end{split}
\end{equation}
Let us denote
\begin{equation}
    R^{\overline{W}}(\overline{\mu},c)=E_{\overline{\mu}}\biggl[ L\biggl(\frac{c}{T^{k(q-1)}}\biggr)\bigg|\overline{W}=\overline{w} \biggr]
\end{equation}
Suppose $\overline{W}=(W_1,W_2,...,W_k)$ is such that $W_i>0$ for all $i'$s and there exists at least one shape parameter $\mu_j>0$, for some $j$. The conditional distribution of $T$ given random vector $\overline{W}$ for $\sigma=1$, is 
\begin{equation}
 f_{T|\overline{W}}(t|\overline{w};\overline{\mu})={\left(1+n\sum_{i=1}^{k}w_{i} \right)}^{kn}\frac{t^{kn-1}}{\Gamma(kn,\eta)}e^{-t(1+n\sum_{i=1}^{k}w_i)}, \text{ \ \ \ where $t>\eta$}.
\end{equation}
Additionally, $R^{\overline{W}}(\overline{\mu},c)$ is a bowl-shaped function of $c$. The value of $c$ for which $R^{\overline{W}}(\overline{\mu},c)$ takes minimum value is $c_{\overline{\mu}}(\overline{w})$ and computed from
\begin{equation} \label{cuw Vali Main equation}
    E_{\overline{\mu}}\biggl[L^{'}\biggl(\frac{c_{\overline{\mu}}(\overline{w})}{T^{k(q-1)}}\biggr) {\frac{1}{T^{k(q-1)}}}\bigg{|}\overline{W}=\overline{w} \biggr]=0.
\end{equation}
Now we prove that $c_{\overline{\mu}}(\overline{w})\le c_{\overline{0}}(\overline{w})$ for $0<q<1$. Suppose that $c_{\overline{\mu}}(\overline{w})>c_{\overline{0}}(\overline{w})$. Since $c_{\overline{0}}(\overline{w})$ is the solution of equation(\ref{cuw Vali Main equation}) for $\overline{\mu}=\overline{0}$ so we can write
\begin{equation}
\begin{split}
    0&=E_{\overline{0}}\biggl[L^{'}\biggl(\frac{c_{\overline{0}}(\overline{w})}{T^{k(q-1)}}\biggr){\frac{1}{T^{k(q-1)}}}\bigg{|}\overline{W}=\overline{w} \biggr] \\
    &<E_{\overline{0}}\biggl[L^{'}\biggl(\frac{c_{\overline{\mu}}(\overline{w})}{T^{k(q-1)}}\biggr){\frac{1}{T^{k(q-1)}}}\bigg{|}\overline{W}=\overline{w} \biggr] \text{ \ \ \ (from the assumption)} \\
    &\le E_{\overline{\mu}}\biggl[L^{'}\biggl(\frac{c_{\overline{\mu}}(\overline{w})}{T^{k(q-1)}}\biggr){\frac{1}{T^{k(q-1)}}}\bigg{|}\overline{W}=\overline{w} \biggr]  \\
    &=0 \text{ \ \ \ (a contradiction)}
\end{split}
\end{equation}
Last inequality follows from the fact that $f_{T|\overline{W}}(t|\overline{w};\overline{\mu})/f_{T|\overline{W}}(t|\overline{w};\overline{0})$, is increasing function in $t$ and utilizing the Lemma(2.2) of {\cite{tripathi2018estimating}}. From above, we can say that 
\begin{equation}\label{cuw is less than c0w}
c_{\overline{\mu}}(\overline{w})\le c_{\overline{0}}(\overline{w}).
\end{equation}

\noindent Now we prove that $c_{\overline{\mu}}(\overline{w})\ge c_{\overline{0}}(\overline{w})$ for $1<q<(n+1)/2$. For $q>1$, the loss function behaves differently. Suppose that $c_{\overline{\mu}}(\overline{w})<c_{\overline{0}}(\overline{w})$. we can write
\begin{equation}\label{showing c0(w)<cu(w) derivation}
\begin{split}
    0&=E_{\overline{0}}\biggl[L^{'}\biggl(\frac{c_{\overline{0}}(\overline{w})}{T^{k(q-1)}}\biggr){\frac{1}{T^{k(q-1)}}}\bigg{|}\overline{W}=\overline{w} \biggr] \\
    &>E_{\overline{0}}\biggl[L^{'}\biggl(\frac{c_{\overline{\mu}}(\overline{w})}{T^{k(q-1)}}\biggr){\frac{1}{T^{k(q-1)}}}\bigg{|}\overline{W}=\overline{w} \biggr] \text{ \ \ \ (from the assumption)} \\
    &\ge E_{\overline{\mu}}\biggl[L^{'}\biggl(\frac{c_{\overline{\mu}}(\overline{w})}{T^{k(q-1)}}\biggr){\frac{1}{T^{k(q-1)}}}\bigg{|}\overline{W}=\overline{w} \biggr]  \\
    &=0 \text{ \ \ \ (a contradiction)}
\end{split}
\end{equation}
In last inequality, we modify the Lemma(2.2) of {\cite{tripathi2018estimating}}, for the loss function of type 
\begin{equation}
L(t)={\left( \frac{1}{t}-1 \right)}^2.
\end{equation}
Now $L^{'}(t)$ is positive for $0<t<1$ and negative for $t>1$. So we can show that for a density function $f(x)$ on $(0,\infty)$, for an increasing non-negative function $g(x)$ on $(0,\infty)$ and for some $x_0\in (0,\infty)$, $h(x)$ changes its sign from positive to negative in $(0,\infty)$, such that $h(x)>0$, for $x<x_0$ and $h(x)<0$, for $x>x_0$ then we can write
\begin{equation}\label{showing the modified lemma}
    0=\int_{0}^{\infty}h(x)g(x)f(x)dx\le \int_{0}^{\infty}h(x)f(x)dx.
\end{equation}
Thus we can say that 
\begin{equation}\label{cuw is less than c0w1}
c_{\overline{\mu}}(\overline{w})\ge c_{\overline{0}}(\overline{w}).
\end{equation}
\noindent Next step is to write $c_{\overline{0}}(\overline{w})$ as a function of $\overline{w}$ for that consider equation(\ref{cuw Vali Main equation}) for $\overline{\mu}=\overline{0}$, we get
\begin{equation}
\begin{split}
    &E_{\overline{0}}\left[L^{'}\left( \frac{c_{\overline{0}}(\overline{w})}{T^{k(q-1)}}\right){\frac{1}{T^{k(q-1)}}}\bigg{|}\overline{W}=\overline{w} \right]=0, \\
    &\bigintsss_{0}^{\infty}L^{'}\left({c_{\overline{0}}(\overline{w})}{t^{k(1-q)}}\right){{t^{k(1-q)}}}{\left(1+n\sum_{i=1}^{k}w_{i} \right)}^{kn}\frac{t^{kn-1}}{\Gamma(kn)}e^{-t(1+n\sum_{i=1}^{k}w_i)}dt=0.
\end{split}
\end{equation}
Substituting $t\left(1+n\sum_{i=1}^{k}w_{i} \right)=z$, we get
\begin{equation}
    \bigintsss_{0}^{\infty}L^{'}\left( \frac{c_{\overline{0}}(\overline{w})z^{k(1-q)}}{{\left(1+n\sum_{i=1}^{k}w_i\right)}^{k(1-q)}}  \right){z^{k(1-q)}}\frac{z^{kn-1}}{\Gamma(kn)}e^{-z}dz=0.
\end{equation}
This gives 
\begin{equation}\label{c1 as a function of w}
c_1=\frac{c_{\overline{0}}(\overline{w})}{{\left(1+n\sum_{i=1}^{k}w_i\right)}^{k(1-q)}}. 
\end{equation}
From the given equation(\ref{c1 is less than c0 in theorem1}), for $0<q<1$, we can write 
\begin{equation}\label{c1 is less than c0}
c_1<c_0,
\end{equation}
also from equation(\ref{c1 is less than c0 in theorem1.2}), for $1<q<(n+1)/2$, we have 
\begin{equation}\label{c1 is greater than c0}
c_1>c_0.
\end{equation}
Define a function as 
\begin{equation}
    \psi^{*}(\overline{w})=\begin{cases}
    \min\Bigl\{ {c_{0}},{c_1}{{(1+n\sum_{i=1}^{k}w_i)}^{k(1-q)}} \Bigr\} ,& \text{if \ \ $0<q<1$ and \ \ } w_{i}>0, \text{\ \ }i=1,2,...,k\\
    \max\Bigl\{ {c_{0}},{c_1}{{(1+n\sum_{i=1}^{k}w_i)}^{k(1-q)}} \Bigr\} ,& \text{if \ \ $1<q<(n+1)/2$ and \ \ } w_{i}>0, \text{\ \ }i=1,2,...,k\\
    
    c_0 ,& \text{otherwise}.
\end{cases}
\end{equation}

For $0<q<1$, from equation(\ref{cuw is less than c0w}), (\ref{c1 as a function of w}) and (\ref{c1 is less than c0}), we have $c_{\overline{\mu}}(\overline{w})<\psi^{*}(\overline{w})<c_{0}$ holds on a set of non-zero probability for $\overline{w}$. Now $R^{\overline{W}}(\overline{\mu},c)$ is a bowl-shaped function of $c$, and when $c>c_{\overline{\mu}}(\overline{w})$, it becomes a increasing function, thus we can write 
\begin{equation} \label{r1(psi)<r1(c0)}
R^{\overline{W}}(\overline{\mu},\psi^{*}(\overline{w}))<R^{\overline{W}}(\overline{\mu},c_{0}). 
\end{equation}
Similar to above if $1<q<(n+1)/2$, from equation(\ref{cuw is less than c0w}), (\ref{c1 as a function of w}) and (\ref{c1 is greater than c0}), we have $c_{0}<\psi^{*}(\overline{w})<c_{\overline{\mu}}(\overline{w})$ holds on a set of non-zero probability for $\overline{w}$. Now $R^{\overline{W}}(\overline{\mu},c)$ is a bowl-shaped function of $c$, and when $c<c_{\overline{\mu}}(\overline{w})$, it becomes a decreasing function, thus we can write 
\begin{equation} \label{r1(psi)<r1(c0)1}
R^{\overline{W}}(\overline{\mu},\psi^{*}(\overline{w}))<R^{\overline{W}}(\overline{\mu},c_{0}). 
\end{equation}
Equation(\ref{r1(psi)<r1(c0)}) and (\ref{r1(psi)<r1(c0)1}), holds for all values of $\overline{\mu}$ and $w_{i}>0,$ $i=1,2,...,k,$ as ${(1+n\sum_{i=1}^{k}w_i)}^{k(1-q)}<{c_0}/{c_1}$, and ${(1+n\sum_{i=1}^{k}w_i)}^{k(1-q)}>{c_0}/{c_1}$ holds respectively depending upon the value of entropy index $q$, on a set of non-zero probability for $\overline{w}$. Hence, we get
\begin{equation}
    R(\overline{\mu},{\delta}_{\psi^{*}})<R(\overline{\mu},{\delta}_{c_0})
\end{equation}

\noindent This completes the proof of the theorem.
\end{proof}
\noindent The existence of a Stein-type estimate for some particular type of $w_i$'s uniformly dominates the BAEE, proves the inadmissibility of the BAEE. The improved estimator lacks smoothness. Next, we discuss the Brewster-Zidek technique which uses sequential approach to reduce the estimator risk.
\section{A Smooth Improved Equivariant Estimate } \label{Smooth improved estimate section4}
\noindent In this section, we derive an estimate of $1/{\sigma}^{k(q-1)}$, which is smooth and improved over $\delta_{c_0}$. The general form of the improved smooth estimator is written as
\begin{equation}\label{equation Brewster general form estimate}
    \delta_{\phi}(X^{(1)},T)=\begin{cases}
      {\frac{\phi(\overline{W})}{{T}^{k(q-1)}}} ,& \text{if } \overline{W} \in B_{r}, \\
    c_0 ,& \text{otherwise}.
    \end{cases}
\end{equation}
where $B_{\textbf{r}}=\times_{i=1}^{k}(0,r_i]$ and $\textbf{r}=(r_1,r_2,...,r_k)$ with $r_i>0,$ for $i=1,2,...,k$. To obtain improved equivariant estimate, we analyze conditional risk 
\begin{equation}
\begin{split}
    R_{1}(\phi,\overline{\mu})=&E_{\overline{\mu}}\biggl[L\biggl({\frac{\phi(\overline{W})}{T^{k(q-1)}}}\biggr)\bigg{|}\overline{W}\in B_{\textbf{r}} \biggr]\\
=& \bigintsss_{0}^{\infty}L\biggl(\frac{\phi(\overline{w})}{t^{k(q-1)}}\biggr)f_{T|B_{\textbf{r}}}(t|\overline{w};\overline{\mu})dt.
    \end{split}
\end{equation}
First we compute the conditional density $f_{T|B_{\textbf{r}}}$, for $\sigma=1$, as the problem of estimation is scale equivariant. Now
\begin{equation}
\begin{split}
    f_{T|B_{\textbf{r}}}(t|\overline{w};\overline{\mu})=&\frac{f_{T,B_{\textbf{r}}}(t,\overline{w};\overline{\mu})}{f_{B_{{\textbf{r}}}}(\overline{w};\overline{\mu})}\\
    =& \frac{\bigintssss_{0}^{r_k} \bigintssss_{0}^{r_{k-1}} \cdots \bigintssss_{0}^{r_1}f_{T,\overline{W}}(t,\overline{w};\overline{\mu})dw_1...dw_{k-1}dw_{k}}{\bigintssss_{0}^{\infty}f_{T,B_{{\textbf{r}}}}(t,\overline{w};\overline{\mu})dt}.
\end{split}
\end{equation}

\noindent The range of random variables $T$, and $\overline{W}$, depends on the nature of the location parameters $u_i$, for $i=1,2,...,k$. Suppose we take all $u_i$ as negative number, then each $W_i$, is from $u_i/T$ to infinity, and T ranges from 0 to infinity. We assume that each $u_i$ is positive then we can write
\begin{equation}
\begin{split}
    f_{T,B_{\textbf{r}}}(t,\overline{w};\overline{\mu})&\propto t^{nk-1}e^{-t}\prod_{i=1}^{k}\int_{0}^{r_i}e^{-{{t}}nw_i}dw_i, \text{ \ \ \ $0<t<\infty$}\\
    &= t^{nk-1}e^{-t}\prod_{i=1}^{k}\int_{\frac{u_i}{t}}^{r_i}e^{-{{t}}nw_i}dw_i \\
    &= t^{k(n-1)-1}e^{-t} \prod_{i=1}^{k} \left( e^{-nu_i}-e^{-tnr_i} \right),
\end{split}
\end{equation}
and marginal density of $\overline{W}$, is given by
\begin{equation}
    f_{B_{{\textbf{r}}}}(\overline{w};\overline{\mu})=\int_{\eta}^{\infty}f_{T,B_{\textbf{r}}}(t,\overline{w};\overline{\mu})dt, \text{ \ \ \ $\overline{w}\in B_{\textbf{r}}.$}
\end{equation}
For a given finite positive integer $k=1, 2, 3, ...$ so on, we can compute above expression. Next lemma is useful in proving the supremacy of $\delta_{\phi}(X^{(1)},T)$ over $\delta_{c_0}(X^{(1)},T)$.

\begin{lemma}
\end{lemma}
    \begin{enumerate}
        \item For each $r_i>0,$ where $i=1,2,...,k,$ $R_{1}(d,\overline{\mu})$ is a strictly U-shaped function in $d$.
        \item Suppose $d({\textbf{r}},\overline{\mu})$ and $d({\textbf{r}},\overline{0})$ are minima of $R_{1}(d,\overline{\mu})$ and $R_{1}(d,\overline{0})$ separately, then $d({\textbf{r}},\overline{\mu})\le d({\textbf{r}},\overline{0})$ holds for $0<q<1$ and $d({\textbf{r}},\overline{\mu})\ge d({\textbf{r}},\overline{0})$ holds for $1<q<(n+1)/2$.
        \item For $0<q<1$, $d({\textbf{r}},\overline{0})$ is non-decreasing function for each $r_i$, where $i=1,2,...,k$.
        \item For $1<q<(n+1)/2$, $d({\textbf{r}},\overline{0})$ is non-increasing function for each $r_i$, where $i=1,2,...,k$.
    \end{enumerate}
\begin{proof}
\begin{enumerate}
        \item \label{lemma4.1(1)} First we prove that $f_{T|B_{{\textbf{r}}}}(t|\overline{w}; \overline{\mu})$ has the monotone likelihood ratio property. Here we transform the random variable $T$ with transformation $Y=\ln(T)$, based on this, the corresponding conditional density function is denoted by $g_{Y|B_{{\textbf{r}}}}(y|\overline{w};\overline{\mu})$ and given by
        \begin{equation}
            g_{Y|B_{{\textbf{r}}}}(y|\overline{w};\overline{\mu})\propto {(e^{y})}^{k(n-1)-1}e^{-e^{y}}e^{y}\prod_{i=1}^{k}\left(1-e^{-nr_i e^{y}} \right), \text{ \ \ \ where $y\in (-\infty,\infty).$}
        \end{equation}
For $d_1<d_2$, consider the ratio
\begin{equation}
\begin{split}
    \frac{g_{Y|B_{{\textbf{r}}}}(y-d_2|\overline{w};\overline{\mu})}{g_{Y|B_{{\textbf{r}}}}(y-d_1|\overline{w};\overline{\mu})} &\propto \frac{{(e^{y-d_2})}^{k(n-1)-1}e^{-e^{y-d_2}}e^{y-d_2}\prod_{i=1}^{k}\left(1-e^{-nr_i e^{y-d_2}} \right)}{{(e^{y-d_1})}^{k(n-1)-1}e^{-e^{y-d_1}}e^{y-d_1}\prod_{i=1}^{k}\left(1-e^{-nr_i e^{y-d_1}} \right)}\\
    &=e^{(k(n-1))(d_1-d_2)}e^{e^{y-d_1}-e^{y-d_2}}\prod_{i=1}^{k}\left(\frac{1-e^{-nr_ie^{y-d_2}}}{1-e^{-nr_ie^{y-d_1}}} \right).
\end{split}
\end{equation}
Substituting $e^y=z$ in above expression and taking $a_0=1,a=nr_i e^{-d_2}, b=nr_i e^{-d_1}$ for any $i=1,2,...,k$, we get $0<a<b$ and then from Lemma(6.1) of {\cite{patra2020estimating}}, we can say that 
\begin{equation*}
    \left(\frac{1-e^{-nr_ie^{y-d_2}}}{1-e^{-nr_ie^{y-d_1}}} \right)
\end{equation*}
is a increasing function of $y$ for all $i=1,2,...,k$. Also it is easy to see that 
\begin{equation*}
    e^{e^{y-d_1}-e^{y-d_2}}=e^{e^y\left(\frac{1}{e^{d_1}}-\frac{1}{e^{d_2}} \right)}
\end{equation*}
is a increasing function in $y$ and thus the considered ratio is increasing in $y$. Also since we have transformed the random variable $T$, we conclude that $f_{T|B_{\textbf{r}}}(t|\overline{w};\overline{\mu})$ has MLR property. From Lemma(2.1) of \cite{brewster1974improving}, it is easy to see that $R_{1}(d,\overline{\mu})$ is U-shaped function in $d$ for every $r_i$.

\item For $0<q<1$, suppose that $d({\textbf{r}},\overline{\mu})> d({\textbf{r}},\overline{0})$. As we know that $d({\textbf{r}},\overline{0})$ is a minima of $R_{1}(d,\overline{0})$ so we are considering
\begin{equation}
\begin{split}
    0&=E_{\overline{0}}\biggl[L^{'}\biggl(\frac{d({\textbf{r}},\overline{0})}{T^{k(q-1)}}\biggr)\frac{1}{T^{k(q-1)}}\bigg{|} \overline{W}\in B_{\textbf{r}}  \biggr]\\
    &<E_{\overline{0}}\biggl[L^{'}\biggl(\frac{d({\textbf{r}},\overline{\mu})}{T^{k(q-1)}}\biggl)\frac{1}{T^{k(q-1)}}\bigg{|} \overline{W}\in B_{\textbf{r}}  \biggl]\\ 
    &\le E_{\overline{\mu}}\biggl[L^{'}\biggl(\frac{d({\textbf{r}},\overline{\mu})}{T^{k(q-1)}}\biggr)\frac{1}{T^{k(q-1)}}\bigg{|} \overline{W}\in B_{\textbf{r}}  \biggr]\\
    &=0  \text{ \ \ \ (a contradiction)}
\end{split}
\end{equation}
Last inequality follows from the fact that the ratio $f_{T|B_{{\textbf{r}}}}(t|\overline{w};\overline{\mu})/f_{T|B_{{\textbf{r}}}}(t|\overline{w};\overline{0})$ is a non-decreasing function in $t$, thus using the Lemma(2.2) of {\cite{tripathi2018estimating}}, we get the desired result.
Also for $1<q<(n+1)/2$, we can follow the same steps as in equation(\ref{showing c0(w)<cu(w) derivation}), utilizing the modified form of Lemma(2.2) given in equation(\ref{showing the modified lemma}), we can get $d({\textbf{r}},\overline{\mu})\ge d({\textbf{r}},\overline{0})$.
\item Now we show that for $0<q<1$, $d({\textbf{r}},\overline{0})$ in $r_i$, for any $i=1,2,...,k$, is non-decreasing. Let $r_i=m$ then ${\textbf{r}}_m=(r_1,r_2,...,m,...,r_k)$. For $m_1<m_2$, consider the function
\begin{equation} \label{increasing behaviour of fTB(t|w)}
    \begin{split}
        \frac{f_{T|B_{{\textbf{r}}_{m_2}}}(t|\overline{w};\overline{0})}{f_{T|B_{{\textbf{r}}_{m_1}}}(t|\overline{w};\overline{0})}\propto&\frac{1-e^{-n{m_2}{t}}}{1-e^{-n{m_1}{t}}}.
    \end{split}
\end{equation}
Substituting $nt=z$, from Lemma(6.2) of {\cite{patra2020estimating}} , we can say that the above ratio of densities is non-increasing in $t$ $i.e.$, $f_{T|B_{{\textbf{r}}_{m_2}}}(t|\overline{w};\overline{0})/f_{T|B_{{\textbf{r}}_{m_1}}}(t|\overline{w};\overline{0})$ is non-increasing in $r_i$. Now we know that $d({\textbf{r}}_{m_1},\overline{0})$ is a minima of $R_{1}(d({\textbf{r}}_{m_1},\overline{0}),\overline{0})$. Suppose $d({\textbf{r}}_{m_1},\overline{0})>d({\textbf{r}}_{m_2},\overline{0}),$ then
\begin{equation} \label{eq for constr in Lemma (1)}
    \begin{split} 0=&E_{\overline{0}}\biggl[L^{'}\biggl(\frac{d({\textbf{r}}_{m_1},\overline{0})}{T^{k(q-1)}}\biggr)\frac{1}{T^{k(q-1)}}\bigg{|}\overline{W}\in B_{{\textbf{r}}_{m_1}} \biggr]\\
    =&\bigintssss_{0}^{\infty}L^{'}\biggl( \frac{d({\textbf{r}}_{m_1},\overline{0})}{t^{k(q-1)}}\biggr)\frac{1}{t^{k(q-1)}}f_{T|B_{{\textbf{r}}_{m_1}}}(t|\overline{w};\overline{0})dt\\
    >& \bigintssss_{0}^{\infty}L^{'}\biggl( \frac{d({\textbf{r}}_{m_{2}},\overline{0})}{t^{k(q-1)}}\biggl)\frac{1}{t^{k(q-1)}}f_{T|B_{{\textbf{r}}_{m_1}}}(t|\overline{w};\overline{0})dt \text{ \ \ (from the assumption)}
    \end{split}
\end{equation}
Consider 
\begin{equation} \label{eq for constr in Lemma (2)}
    \begin{split}
0=&E_{\overline{0}}\biggr[L^{'}\biggr(\frac{d({\textbf{r}}_{m_2},\overline{0})}{T^{k(q-1)}}\biggl)\frac{1}{T^{k(q-1)}}\bigg{|}\overline{W}\in B_{{\textbf{r}}_{m_2}} \biggl] \\
&=\bigintssss_{0}^{\infty}L^{'}\biggl( \frac{d({\textbf{r}}_{m_2},\overline{0})}{t^{k(q-1)}}\biggr)\frac{1}{t^{k(q-1)}}f_{T|B_{{\textbf{r}}_{m_2}}}(t|\overline{w};\overline{0})dt\\
&=\bigintssss_{0}^{\infty}L^{'}\biggl( \frac{d({\textbf{r}}_{m_2},\overline{0})}{t^{k(q-1)}}\biggr)\frac{1}{t^{k(q-1)}}\frac{f_{T|B_{{\textbf{r}}_{m_2}}}(t|\overline{w};\overline{0})}{f_{T|B_{{\textbf{r}}_{m_1}}}(t|\overline{w};\overline{0})}f_{T|B_{{\textbf{r}}_{m_1}}}(t|\overline{w};\overline{0})dt\\
&\le \bigintssss_{0}^{\infty}L^{'}\biggl( \frac{d({\textbf{r}}_{m_2},\overline{0})}{t^{k(q-1)}}\biggr)\frac{1}{t^{k(q-1)}}f_{T|B_{{\textbf{r}}_{m_1}}}(t|\overline{w};\overline{0})dt.
    \end{split}
\end{equation}
Last step follows from equation(\ref{increasing behaviour of fTB(t|w)}) and the Lemma(1.2) of {\cite{kanta2020improved_for_lemma_of_ratio_of_densities}}. Combining equations(\ref{eq for constr in Lemma (1)}) and (\ref{eq for constr in Lemma (2)}), we get a contradiction. This prove that $d({\textbf{r}}_{m_1},\overline{0})\le d({\textbf{r}}_{m_2},\overline{0})$. That is, $d({\textbf{r}}_,\overline{0})$ is a non-decreasing function for each $r_i$, $i=1,2,...,k$.
\item Similar to Lemma(1.2) of \cite{kanta2020improved_for_lemma_of_ratio_of_densities}, for $1<q<(n+1)/2$, if a density function $f(x)$ defined on $(0,\infty)$, a non-negative non-increasing function $g(x)$ on $(0,\infty)$ and for some $x_0$ in $(0,\infty)$, $h(x)$ changes its sign from positive to negative such that $h(x)>0$, for $x<x_0$ and $h(x)<0$ for $x>x_0$, then we can write the inequality
\begin{equation}\label{modified 2}
        0=\int_{0}^{\infty}h(x)g(x)f(x)dx\le \int_{0}^{\infty}h(x)f(x)dx.
\end{equation}
It is easy to see that following the steps of equation(\ref{eq for constr in Lemma (1)}), (\ref{eq for constr in Lemma (2)}) and applying equation(\ref{modified 2}), we can show that for $1<q<(n+1)/2$, $d({\textbf{r}}_,\overline{0})$ is a non-increasing function for each $r_i$, $i=1,2,...,k$.

\end{enumerate}
\end{proof}

\begin{theorem} \label{dominance theorem1}
    For the quadratic error function $L(t)$, the risk value of the estimator $\delta_{\phi}(X^{(1)},T)$, is nowhere greater than the risk value of the estimator $\delta_{c_0}(X^{(1)},T)$, where
    \begin{equation}
        \phi({\overline{w}})=\begin{cases}
            d({\textbf{r}},\overline{0}), & \overline{w}\in B_{\textbf{r}}, \\
            c_0, & otherwise.
        \end{cases}
    \end{equation}
\end{theorem}
\begin{proof}
    We can write the risk value of an estimator $\delta_{\phi}(X^{(1)},T)$, as
    \begin{equation} \label{risktheorem1 R1}
    \begin{split}
        R(\overline{\mu},\delta_{\phi}(X^{(1)},T))=E_{\overline{\mu}}\left[L\left({d({\textbf{r}},\overline{0})}{T^{k(1-q)}}\right)\bigg{|}\overline{W}\in B_{{\textbf{r}}} \right]P\left(\overline{w}\in B_{\textbf{r}} \right)+ \\ E_{\overline{\mu}}\left[L\left({c_0}{T^{k(1-q)}}\right)\bigg{|}\overline{W}\notin B_{{\textbf{r}}} \right]P\left(\overline{w}\notin B_{\textbf{r}} \right).
        \end{split}
    \end{equation}
    If $0<q<1$ and $\overline{w}\in B_{\textbf{r}}$, $d({\textbf{r}},\overline{w})\le d({\textbf{r}},\overline{0})\le c_0$, and since $R_{1}(d,\overline{\mu})$ is a bowl-shaped function, so we get
    \begin{equation}\label{risktheorem1 r2}
        E_{\overline{\mu}}\left[L\left({d({\textbf{r}},\overline{0})}{T^{k(1-q)}}\right)\bigg{|}\overline{W}\in B_{{\textbf{r}}} \right] \le 
        E_{\overline{\mu}}\left[L\left({c_0}{T^{k(1-q)}}\right)\bigg{|}\overline{W}\in B_{{\textbf{r}}}  \right]. 
    \end{equation}
    Or if $1<q<(n+1)/2$ and $\overline{w}\in B_{\textbf{r}}$, $d({\textbf{r}},\overline{w})\ge d({\textbf{r}},\overline{0})\ge c_0$, and since $R_{1}(d,\overline{\mu})$ is a bowl-shaped function, so we get
    \begin{equation}\label{risktheorem1 r2 2}
        E_{\overline{\mu}}\left[L\left({d({\textbf{r}},\overline{0})}{T^{k(1-q)}}\right)\bigg{|}\overline{W}\in B_{{\textbf{r}}} \right] \le 
        E_{\overline{\mu}}\left[L\left({c_0}{T^{k(1-q)}}\right)\bigg{|}\overline{W}\in B_{{\textbf{r}}}  \right]. 
    \end{equation}
    From equation(\ref{risktheorem1 R1}), (\ref{risktheorem1 r2}) and (\ref{risktheorem1 r2 2}), we concluded that 
    \begin{equation}
        R(\overline{\mu},\delta_{\phi}(X^{(1)},T))\le R(\overline{\mu},\delta_{c_0}(X^{(1)},T)).
    \end{equation}
    This completes the proof.
\end{proof}
\noindent We will now illustrate the existence of an estimate for finite values of $k(=1,2)$, in the next example, assuming location parameters are positive. 

\begin{example} \label{example1}
    Consider a sample of exponentially distributed population $\exp(\sigma, \mu)$, the equivariant estimate $\delta_{\phi}(X^{(1)},T)$ has the risk value nowhere greater than the risk value of estimator BAEE, $\delta_{c_0}(X^{(1)},T)$ for the quadratic error function $L(t)$, and any real number $r_1 >0$, where 
    \begin{equation}
        \phi(w)=\begin{cases}
            \frac{\Gamma(n-1+(1-q))}{\Gamma(n-1+2(1-q))}\left(\frac{1-{(1+nr_1)}^{q-n}}{1-{(1+nr_1)}^{2q-n-1}}\right), &  \text{\ \ \  if \ } w\in B_{r_1}, \\
            c_0, & \text{ \ \  otherwise.} 
        \end{cases}
    \end{equation}
\end{example}
\begin{Solution}
    As we can see that this case is for $k=1$, and we know that the value of $d_{r_1},$ at which $R_1(d,\overline{0})$ 
    is minimum given by
    \begin{equation}
    \begin{split}
        d(r_1,0)=&\frac{E_{0,1}\left( T^{(1-q)}| W\in B_{r_1} \right)}{E_{0,1}\left( T^{2(1-q)}| W\in B_{r_1} \right)}\\
        =& \frac{\int_{0}^{\infty}t^{1-q}f_{T|B_{r_1}}(t|w)dt}{\int_{0}^{\infty}t^{2(1-q)}f_{T|B_{r_1}}(t|w)dt}\\
        =&\frac{\Gamma(n-1+1-q)}{\Gamma(n-1+2(1-q))}\left(\frac{1-{(1+nr_1)}^{q-n}}{1-{(1+nr_1)}^{2q-n-1}}\right).
        \end{split}
    \end{equation}
    Dominance of improved smooth estimator can be seen from Theorem(\ref{dominance theorem1}). 
\end{Solution}

\begin{Note}
    It can be seen that if $r_1 \rightarrow \infty$, we have
    \begin{equation}
        {(1+nr_1)}^{q-n}\rightarrow 0 \text{ \ \ and } {(1+nr_1)}^{2q-n-1}\rightarrow 0.
    \end{equation}
    Thus we get, $\lim_{r_1 \rightarrow \infty}d(r_1,0)=c_0$.
\end{Note}
\begin{example}\label{example2}
    Consider two samples from exponentially distributed population $\exp(\sigma, \mu_1)$ and $\exp(\sigma, \mu_2)$, the equivariant estimate $\delta_{\phi}(X^{(1)},T)$ has the risk value nowhere greater than the risk value of estimator BAEE $\delta_{c_0}(X^{(1)},T),$ for the quadratic error function $L(t)$, and any real number $r_1,r_2 >0$, where 
    \begin{equation}
        \phi(\overline{w})=\begin{cases}
            \frac{\Gamma(2(n-1)+(1-q))}{\Gamma(2(n-1)+2(1-q))}\left(\frac{1-{(1+nr_1)}^{q+1-2n}-{(1+nr_2)}^{q+1-2n}+{(1+nr_1+nr_2)}^{q+1-2n}}{1-{(1+nr_1)}^{2q-2n}-{(1+nr_2)}^{2q-2n}+{(1+nr_1+nr_2)}^{2q-2n}}\right), &  \text{\ \ \  if \ } (w_1,w_2)\in B_{{\textbf{r}}}, \\
            c_0, & \text{ \ \  otherwise.} 
        \end{cases}
    \end{equation}
\end{example}
\begin{Solution}
    This case is when $k=2$ and ${\textbf{r}}=(r_1,r_2)$, Following the same steps of Example(\ref{example1}), we get the desired results.
\end{Solution}
\begin{Note}
    It can be seen that if $r_1$, and $r_2$ tends to infinity, we get
    \begin{equation}
        \lim_{r_1\rightarrow \infty, r_2 \rightarrow \infty}d({\textbf{r}},\overline{0})=c_0.
    \end{equation}
\end{Note}
\noindent Now consider the vector ${\textbf{r}}^{1}=(r_{11},r_{12},...,r_{1k})$ such that $0<r_{1i}<r_{i}$ for $i=1,2,...,k$. let $B_{{\textbf{r}}^{1},{\textbf{r}}}=\times_{i=1}^{k}(r_{1i},r_i]$. Let us define the estimate in the form 
\begin{equation}
    \delta_{d,{\textbf{r}},{\textbf{r}}^{1}}(X^{(1)},T)=\begin{cases}
      \frac{d({\textbf{r}}^{1},\overline{0})}{{T}^{k(q-1)}} , & \text{if } \overline{W} \in B_{{\textbf{r}}^{1}} \\
      \frac{d({\textbf{r}},\overline{0})}{{T}^{k(q-1)}} ,& \text{if } \overline{W} \in B_{{\textbf{r}}^{1},{\textbf{r}}} \\
    c_0 ,& \text{otherwise}.
    \end{cases}
\end{equation}
We just need to find $d({\textbf{r}}^{1},\overline{0})$, for this general form and similar to Theorem(\ref{dominance theorem1}), we can state that $\delta_{\phi_{{\textbf{r}}^{1},{\textbf{r}}}}(X^{(1)},T)$ has risk with respect to loss function $L(t)$, less than or equal to $\delta_{c_0}(X^{(1)},T)$, where 
\begin{equation}
    \phi_{{\textbf{r}}^{1},{\textbf{r}}}=\begin{cases}
        d({\textbf{r}}^{1},\overline{0}), & \text{ \ \ \ if } \overline{w}\in B_{\textbf{r}} \\
        d({\textbf{r}},\overline{0}), & \text{ \ \ \ if } \overline{w}\in B_{{\textbf{r}}^{1},{\textbf{r}}} \\
        c_0, & \text{ \ \ \ otherwise. } 
    \end{cases}
\end{equation}
Consider the partition of set $\times_{i=1}^{k}(0,r_i]$ as $P_{ij}=\{ 0=r_{j,0}^{i},r_{j,1}^{i},...,r_{j,m_j}^{i}  \},$ for $i=1,2,...,k$ and $j$ is denoting the partition number and $m_j$ is the number of points in the partition. This can be seen in {\cite{patra2023inadmissibilityArxiv}}. For better clarity,
\begin{equation*}
    \begin{cases}
        0=r_{j,0}^{1}<r_{j,1}^{1}<r_{j,2}^{1}<\cdots<r_{j,m_j}^{1}<\infty \\
        0=r_{j,0}^{2}<r_{j,1}^{2}<r_{j,2}^{2}<\cdots<r_{j,m_j}^{2}<\infty \\
        \vdots \\
        0=r_{j,0}^{k}<r_{j,1}^{k}<r_{j,2}^{k}<\cdots<r_{j,m_j}^{k}<\infty.
    \end{cases}
\end{equation*}
Take ${\textbf{r}}_{j,l}=(r_{j,l}^{1},r_{j,l}^{2},...,r_{j,l}^{k})$, where $l=1,2,...,m_j$ and denote
\begin{equation}
    \phi_{j}(\overline{w})=\begin{cases}
        d({\textbf{r}}_{j,1},\overline{0}), & \text{ \ \ if } r_{j,0}^{i}<w_i\le r_{j,1}^{i},i=0,...,k \\ 
        d({\textbf{r}}_{j,2},\overline{0}), & \text{ \ \ if } r_{j,1}^{i}<w_i\le r_{j,2}^{i},i=0,...,k \\ 
        d({\textbf{r}}_{j,3},\overline{0}), & \text{ \ \ if } r_{j,2}^{i}<w_i\le r_{j,3}^{i},i=0,...,k \\ 
        \vdots \\
        d({\textbf{r}}_{j,m_j},\overline{0}), & \text{ \ \ if } r_{j,m_j-1}^{i}<w_i\le r_{j,m_j}^{i},i=0,...,k \\ 
        c_0, & \text{ \ \  otherwise.}
    \end{cases}
\end{equation}
Assuming that the partition is such that 
\begin{equation}
    \max_{1\le k \le m_j}|r_{j,k}^{v}-r_{j,k-1}^{v}|\rightarrow 0 \text{ \ \ as \ \ } j\rightarrow\infty \text{ \ \ for \ \ } v=1,2,...,k.
\end{equation}
Because of this, $\phi_{j}(\overline{w})\rightarrow \Phi^{*}(\overline{w})$ pointwise as $j\rightarrow\infty$. From Theorem(\ref{dominance theorem1}), we know that for each $j$, $\delta_{\phi_{j}}$ is dominating the estimate $\delta_{c_0}$, and from Fatou's lemma, we get the following 
theorem as a result.
\begin{theorem} \label{ dominance theorem 2}
    For the quadratic error function $L(t)$, the risk value of the estimator $\delta_{\Phi^{*}}(X^{(1)},T)$, is nowhere greater than the risk value of the estimator $\delta_{c_0}(X^{(1)},T)$, where
    \begin{equation}
        \Phi^{*}({\overline{w}})=\begin{cases}
            d(\overline{w},\overline{0}), & w_1>0,w_2>0,...,w_k>0, \\
            c_0, & otherwise.
        \end{cases}
    \end{equation}
\end{theorem}
\begin{proof}
We can write from the Theorem(\ref{dominance theorem1}) that 
\begin{equation} \label{first equation of dominance theorem 2}
    R(\overline{\mu},\delta_{\phi_{j}}(X^{(1)},T))\le R(\overline{\mu},\delta_{c_0}(X^{(1)},T)), \text{ \ \ \ } j=1,2,... 
\end{equation}  
Since $\phi_{j}(\overline{w})\rightarrow \Phi^{*}(\overline{w})$ as $j\rightarrow \infty$, so from Fatou's lemma we can write
\begin{equation} \label{2nd equation of dominance theorem 2}
    R(\overline{\mu},\delta_{\Phi^{*}}(X^{(1)},T)) \le \liminf_{j \rightarrow \infty} R(\overline{\mu},\delta_{\phi_{j}}(X^{(1)},T)).
\end{equation}
Combining equations(\ref{first equation of dominance theorem 2}), (\ref{2nd equation of dominance theorem 2}), we can write
\begin{equation}
    R(\overline{\mu},\delta_{\Phi^{*}}(X^{(1)},T)) \le \liminf_{j \rightarrow \infty} R(\overline{\mu},\delta_{\phi_{j}}(X^{(1)},T)) \le R(\overline{\mu},\delta_{c_0}(X^{(1)},T)),
\end{equation}
gives $R(\overline{\mu},\delta_{\Phi^{*}}(X^{(1)},T))\le R(\overline{\mu},\delta_{c_0}(X^{(1)},T))$, which completes the proof.
\end{proof}
\noindent 
Referring to the previously mentioned examples(\ref{example1}) and (\ref{example2}), the above theorem's applicability becomes more clear.
\begin{example} \label{example3}
    For a sample of exponentially distributed population $\exp(\sigma, \mu)$, the equivariant estimate $\delta_{\Phi^{*}}(X^{(1)},T)$ has the risk value nowhere greater than the risk value of estimator BAEE $\delta_{c_0}(X^{(1)},T)$, subject to the quadratic error function $L(t)$, where 
    \begin{equation}
        \Phi^{*}(w_1)=\begin{cases}
            \frac{\Gamma(n-1+(1-q))}{\Gamma(n-1+2(1-q))}\left(\frac{1-{(1+nw_1)}^{q-n}}{1-{(1+nw_1)}^{2q-n-1}}\right), &  \text{\ \ \  if \ } w_1\ge 0, \\
            c_0, & \text{ \ \  otherwise.} 
        \end{cases}
    \end{equation}
\end{example}
\begin{example}
    For the two samples of exponentially distributed population $\exp(\sigma, \mu_1)$ and $\exp(\sigma, \mu_2)$, the equivariant estimate $\delta_{\Phi^{*}}(X^{(1)},T)$, has the risk value nowhere greater than the risk value of estimator BAEE $\delta_{c_0}(X^{(1)},T)$, subject to the quadratic error function $L(t)$, where 
    \begin{equation}
        \Phi^{*}(\overline{w})=\begin{cases}
            \frac{\Gamma(2(n-1)+(1-q)}{\Gamma(2(n-1)+2(1-q))}\left(\frac{1-{(1+nw_1)}^{q+1-2n}-{(1+nw_2)}^{q+1-2n}+{(1+nw_1+nw_2)}^{q+1-2n}}{1-{(1+nw_1)}^{2q-2n}-{(1+nw_2)}^{2q-2n}+{(1+nw_1+nw_2)}^{2q-2n}}\right), &  \text{\ \ \  if \ } w_1\ge 0, w_2 \ge 0, \\
            c_0, & \text{ \ \  otherwise.} 
        \end{cases}
    \end{equation}
    
\end{example}
\subsection{Bayes Estimator}
In this section, we look at the Bayesian estimate of $\Theta(\sigma)$, for the quadratic loss function given in (\ref{lossfunction}). Based on the existing literature, we take inverse gamma distributions, denoted as $IG(\alpha,\beta)$ with density 
\begin{equation}\label{inverse gamma density function}
    g_{\alpha,\beta}(\sigma)=\frac{{\alpha}^{\beta}}{\Gamma(\beta)}{\left(\frac{1}{\sigma} \right)}^{\beta+1}e^{-\left(\frac{\alpha}{\sigma} \right)}, \textit{ \ \ \ }\sigma>0, \ \ \alpha>0, \ \ \beta>0,
\end{equation}
for the parameter $\sigma$ as a prior distribution. Using the transformation of random variables, the density function for the random variable $T$, conditional on $\sigma$ values, is
\begin{equation}
    f(t|\sigma)=\frac{t^{k(n-1)-1}e^{-\frac{t}{\sigma}}}{{\sigma}^{k(n-1)}\Gamma(k(n-1))}, \textit{ \ \ \ where \ \  $0<t<\infty.$}
\end{equation}
Now the posterior density using Bayes formula is,
\begin{equation}\label{bayesformula}
    h(\sigma|t)=\frac{f(t|\sigma)g_{\alpha,\beta}(\sigma)}{f(t)}.
\end{equation}
where the probability density function $f(t)$ of $T$, independent of $\sigma$, is computed as
\begin{equation*}
    f(t)=\int_{0}^{\infty}f(t|\sigma)g_{\alpha,\beta}(\sigma)d{\sigma}=\frac{{\alpha}^{\beta}x^{k(n-1)-1}}{\Gamma(\beta)\Gamma(k(n-1))}x^{-(k(n-1)+\beta)}\Gamma(k(n-1)+\beta).
\end{equation*}
Using the (\ref{bayesformula}), the posterior density is given by
\begin{equation}\label{posterior density}
    h(\sigma|t)=\frac{{\sigma}^{-(k(n-1)+\beta+1)}e^{-\left(\frac{x+\alpha}{\sigma} \right)}}{{(x+\alpha)}^{-(k(n-1)+\beta)}\Gamma(k(n-1)+\beta)}.
\end{equation}
As (\ref{lossfunction}), is the quadratic error function, the Bayes estimator for $\Theta(\sigma)$, is computed using the relation obtained after minimizing the corresponding risk value, we get
\begin{equation}\label{bayes estimate 1}
\begin{split}
    {\delta}_{\alpha,\beta}^{B}(T)&=\frac{{E({\sigma}^{k(1-q)}}|T)}{{E({\sigma}^{2k(1-q)}}|T)} \\
   &=\frac{1}{{(T+\alpha)}^{k(q-1)}}\frac{\Gamma(k(n-1)+k(1-q)+\beta)}{\Gamma(k(n-1)+2k(1-q)+\beta)}.
    \end{split}
\end{equation}
which is the Bayesian estimate based on the prior inverted gamma distribution $IG(\alpha,\beta)$ and new information from the sample of size $n$ from $k$ many exponentially distributed populations each with different location parameter $\mu_i, i=1,2,...,k$. Bayesian framework incorporate many estimation techniques as a particular case. 

\begin{remark}
    For $\alpha\to 0$ and $\beta\to 0$, Bayesian estimate behaves same as the BAEE given in Theorem(\ref{BAEE}).
\end{remark}

\section{Numerical Comparisons}\label{numerical comparisons section6}
In the previous sections, the derivation of the Stein's estimator and Brewster-Zidek type estimator of the Tsallis entropy for several exponentially distributed population is presented. In this section, we test the performance of the derived estimators through simulations using the 'rtpexp' function from the 'twopexp' package in RStudio, for a range of sample sizes. Percentage risk improvement(PRI) is one of the measure used for comparison of the estimators. The PRI of the estimator $\delta_2$ relative to the estimator $\delta_1$ is given by
\begin{equation}\label{PRIformula}
    PRI(\delta_2)=\frac{Risk(\delta_1)-Risk(\delta_2)}{Risk(\delta_1)}\times100.
\end{equation}

\begin{figure}[ht!]
\centering
\includegraphics[width=110mm]{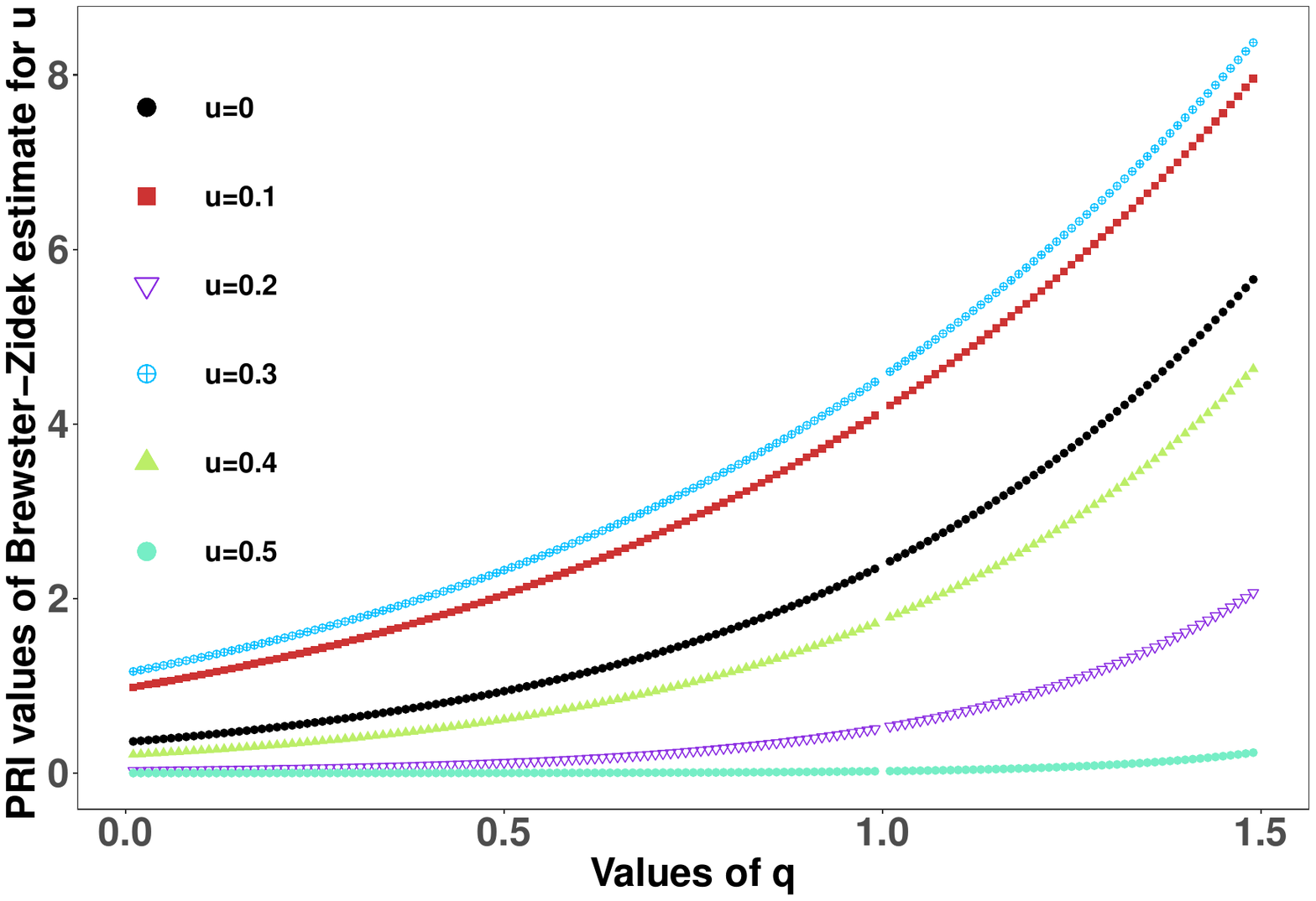}
\caption{PRI values of Brewster-Zidek estimate for a simulated sample of size $n=4, \sigma=1$, $u=0,0.1,0.2,0.3,0.4,0.5$ and $q>0.$  \label{figure1}}
\end{figure}

\begin{figure}[ht!]
\centering
\includegraphics[width=115mm]{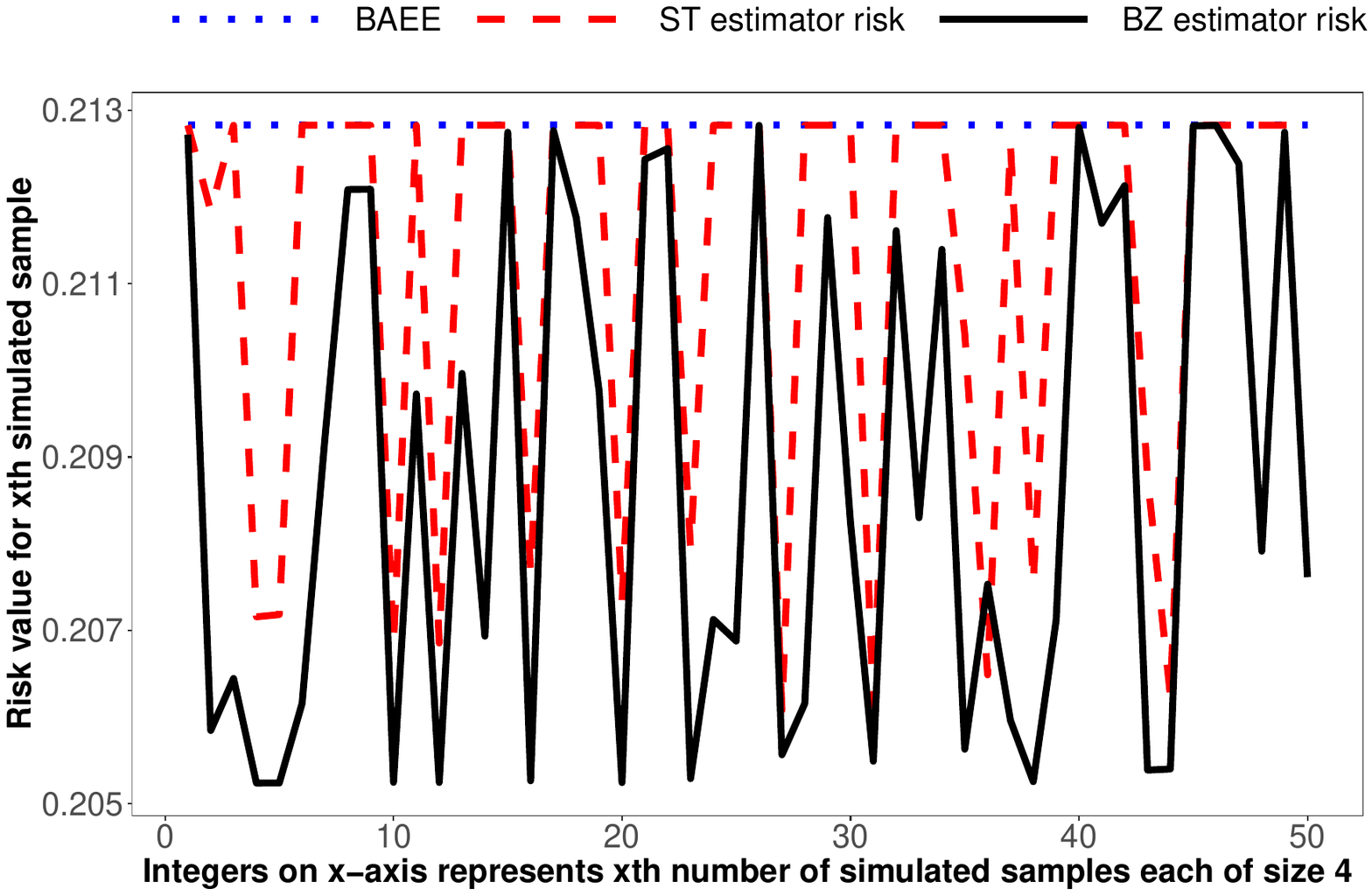}
\caption{Risk value plot of BAEE, Stein-type and Brewster-Zidek estimator for $\sigma=1, u=0.1, q=0.1$ and $50$ samples of sizes $4$  \label{figure3}}
\end{figure}

\begin{figure}[ht!]
\centering
\includegraphics[width=115mm]{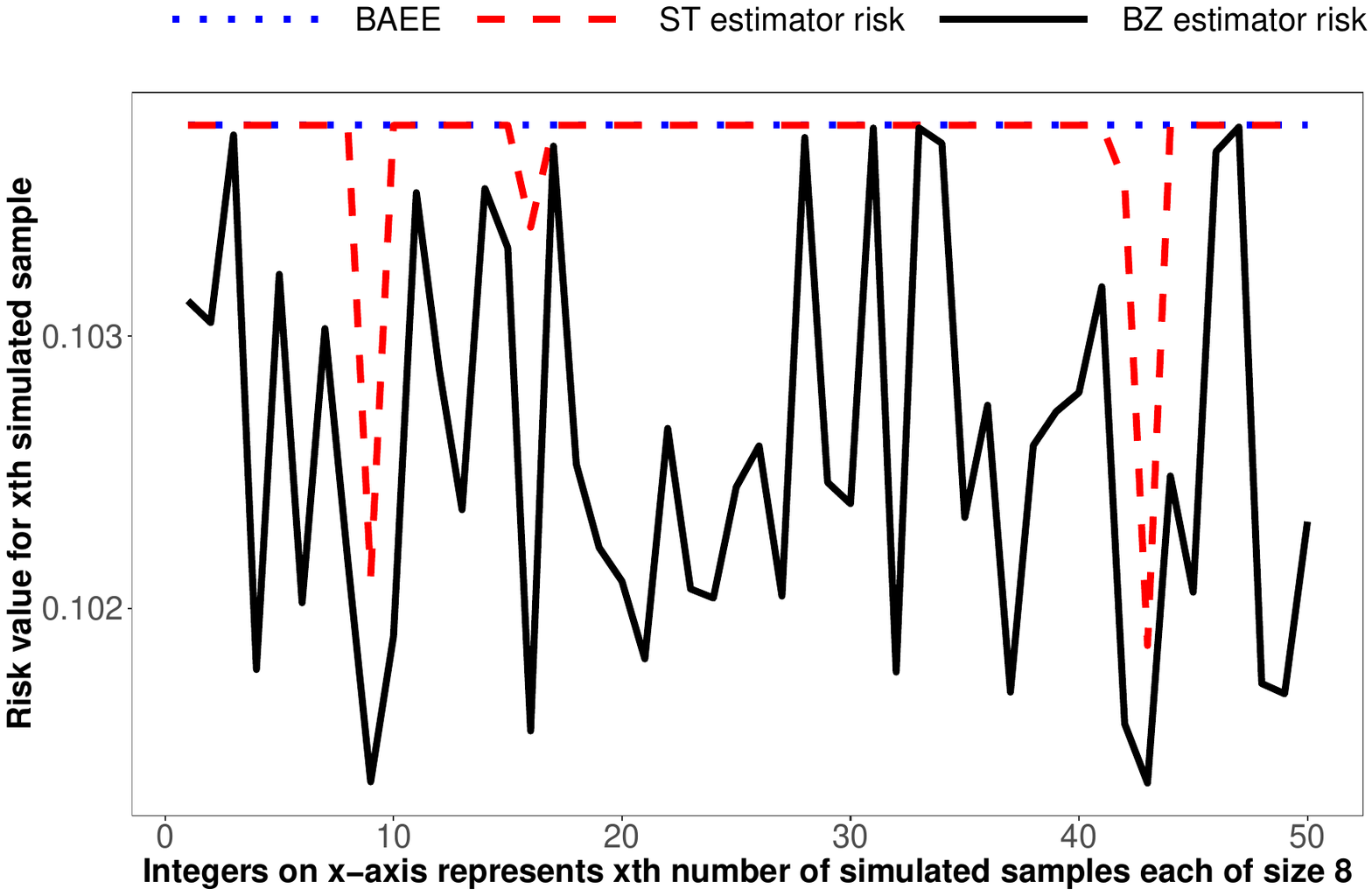}
\caption{Risk value plot of BAEE, Stein-type and Brewster-Zidek estimator for $\sigma=1, u=0.1, q=0.1$ and $50$ samples of sizes $8$  \label{figure4}}
\end{figure}
\begin{figure}[ht!]
\centering
\includegraphics[width=115mm]{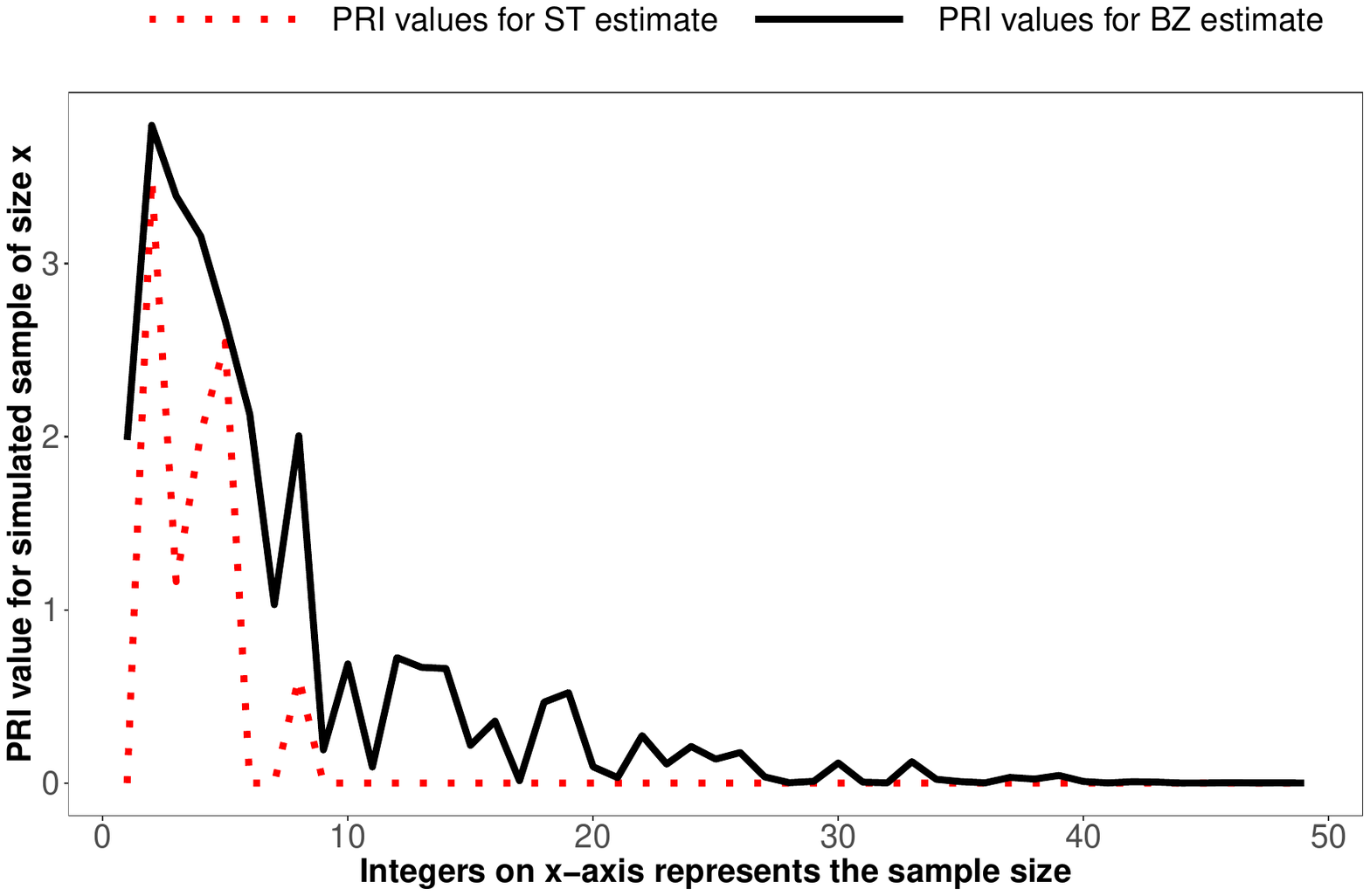}
\caption{PRI values plot of Stein-type and Brewster-Zidek estimator for $\sigma=1, u=0.1, q=0.1$ and sample size vary from $2$ to $50$  \label{figure5}}
\end{figure}

For a single population with $\sigma=1, n=4,6,8$ and varying values of location parameter $u$, over $10000$ random samples were generated. Table(\ref{table1}) in Appendix, presents the PRI values for the indicated parameters, calculated using the simulated sample corresponding to the specified parameter values. We calculated the PRIs for various values of the entropic index parameter $q$ by altering the parameter $q$, for a single sample corresponding to the specified parameter. In Table(\ref{table2}) in Appendix, we calculated the PRI values for estimates with large sample sizes($n=10,15,20$ and $30$). The tables of PRI values are presented in the appendix for reference. We use $\delta_{ST}$ and $\delta_{BZ}$ to denote the Stein-type estimator and Brewster-Zidek type estimate, respectively. Some of the observations are as follows:
\begin{enumerate}
    \item The presence of more positive PRI for $\delta_{BZ}$ than for $\delta_{ST}$ implies that the Brewster estimate provides greater estimate improvement. 
    \item The PRI for $\delta_{ST}$ is zero for a greater number of parameters, suggesting that $\delta_{ST}$ has not significantly improved.
    \item As the sample size $n$ increases, the PRI become zero, signifying improved estimates are converging to the BAEE.
\end{enumerate}

\noindent PRI values of a simulated samples are computed for parameters $\sigma=1,n=4,$ and $q$ ranging from $0<q<1.5$, can be seen in Figure(\ref{figure1}). Figure illustrates the variation in PRI values for Brewster-Zidek improved estimator for different values of $u=0,0.1,0.2,0.3,0.4$, and $0.5$. It is clear from the Figure(\ref{figure1}), increase in value of $q$, increases the PRI value for Brewster-Zidek estimate. In the figure, values corresponding to $q=1$ are omitted. Based on the simulated samples, it can be observed that there is no definitive relation between PRI of Brewster-Zidek estimate and location parameter values $u$. This observation is supported by the analysis of graphs corresponding to $u=0,0.1,$ and $0.4$.

To further clarify the behaviour of risk, $50$ samples were generated for parameters $\sigma=1,u=0.1,q=0.1$ with a sample size of $4$ and $8$. The corresponding risk values for both sample sizes are depicted in Figure(\ref{figure3}) and Figure(\ref{figure4}). The selection of parameters and the sample numbers is motivated by the intention to present a clear illustration of risk behaviour and choice for a higher number of samples may not yield a visually clear representation. The dotted straight line at the top represents the constant risk associated with the BAEE. The lower values of risk associated with the Brewster-Zidek estimator in Figure(\ref{figure3}) clearly indicate a significant improvement over the BAEE. Although, the Stein-type improved estimator exhibits a lower risk value, it can be seen that the Brewster-Zidek improved estimate dominates in terms of achieving an even lower risk value.

The effect of sample size on the risk value of estimates can be seen in the Figure(\ref{figure5}). We plotted PRI values for both Stein-type and Brewster-type estimates as the generated sample size increases from $2$ to $50$. It is readily noticeable that the risk value decreases for both estimates as the sample size increases. Upon reaching an adequate sample size, it is observed that the risk value is the same for Stein-type and Brewster estimates, which is equal to risk value for BAEE.

\section{Concluding Remarks}\label{conclusion section7}
We have computed the Tsallis entropy for $k$ number of independently distributed exponential populations, with a common scale parameter and distinct location parameters, for $q>0$. We considered the problem of estimating the Tsallis entropy for $k$ independently distributed exponential populations, with a common scale parameter and distinct location parameters, for $q>0$. The BAEE for the function of scale parameter associated with Tsallis entropy, is derived under the strictly bowl-shaped quadratic loss function with restrictions on $q$. Further, the inadmissibility of the BAEE is established by computing the existence of the Stein-type estimate. A class of smooth improved estimates for the parameter function based on Brewster technique is provided. It is shown that the Bayesian estimates of the parameter function with inverse gamma priors, under the quadratic loss function, provides the BAEE. A simulation study is conducted using specific parameter values to illustrate the performance of estimators. It is noted that the PRI value increases with increases in $q$. The table(\ref{table1}) demonstrate that the improved estimates approach the BAEE as the sample size increases. The entire analysis can also be conducted with any bowl-shaped loss function, such as Linex and entropic loss functions.

\section*{Appendix}

\begin{table}[!h] 
\caption { PRI values of $\delta_{ST}$ and $\delta_{BZ}$ for generated samples of sizes $n=4,6,8$ with parameters $\sigma=1,u=0.1,0.2,0.3,0.4,0.5,0.6$ and $q=0.2,0.4,0.6,0.8,1.2,1.4$ } \label{table1}
\begin{center}
{\begin{tabular}{ >{\centering\arraybackslash}p{1cm} >{\centering\arraybackslash}p{1cm}>{\centering\arraybackslash}p{1.8cm}>{\centering\arraybackslash}p{1.8cm}>{\centering\arraybackslash}p{1.8cm}>{\centering\arraybackslash}p{1.8cm}>{\centering\arraybackslash}p{1.8cm}>{\centering\arraybackslash}p{1.8cm}}
\hline
\multicolumn{2}{c}{n} &\multicolumn{2}{c}{4}&\multicolumn{2}{c}{6}&\multicolumn{2}{c}{8}\\
\hline
  $u$ & $q$ & $\delta_{ST}$ &$\delta_{BZ}$ & $\delta_{ST}$ &$\delta_{BZ}$  & $\delta_{ST}$ &$\delta_{BZ}$ \\ 

\cline{1-8}
 \multirow{ 8}{*}[15pt]{ $0.1$} & $0.2$&$ 0.168570$ & $3.469226$ &$2.335328$  &$2.988272$ 
 &$0.727848$  &$2.281872$ \\

 &$0.4$ & $1.087688$  & $4.129227$ &
 $2.665745$ & $3.283123$ &
 $0.985590$ & $2.487200$  \\

 &$0.6$ & $2.118459$ & $4.890057$ &
 $3.016849$ & $3.592175$ &
 $1.257724$ & $2.701414$\\

&$0.8$  & $3.267641$ & $5.760361$ &
$3.386822$ & $3.913180$ 
& $1.542669$ & $2.923200$ \\

 &$1.2$  & $0$ & $7.851997$ &
 $0$ & $4.577070$ &
 $0$  & $3.381832$ \\

 &$1.4$  & $0$ & $9.071442$  &
 $0$ & $4.909568$ &
 $0$ & $3.613339$ \\
 \cline{1-8}

 
 \multirow{ 8}{*}[15pt]{ $0.2$ }&$0.2$ &$1.998670$ & $3.746505$ & 
 $0$ & $0.380079$ & 
 $0$ & $0.767991$ \\

 & $0.4$& $2.751726$ & $4.371731$ &
 $0$ & $0.497380$ & 
 $0$  & $0.899383$ \\

 &$0.6$ & $3.597183$ & $5.076163$ &
 $0$ & $0.647878$ & 
 $0$ &$1.049318$ \\

& $0.8$ & $4.540664$ & $5.863484$ &
$0$ & $0.839757$ &
$0$ &$1.219447$ \\

 & $1.2$ & $0$ & $7.687807$ & 
 $6.855561$ & $1.387746$ &
 $4.695817$ & $1.626045$ \\

 &$1.4$  & $0$  & $8.712302$ &
 $7.333795$ & $1.767291$ &
 $4.613808$ & $1.864516$ \\
 \cline{1-8}

 \multirow{ 8}{*}[15pt]{ $0.3$  }&$0.2$   & $0$ & $3.119990$ &
 $0$ & $2.080602$ & 
 $0$ & $0.002599$ \\

 &$0.4$ & $0$ &$3.783965$ &
 $0$ & $2.411114$ & 
 $0$ & $0.003937$ \\

 &$0.6$ & $0.097226$ &$4.565649$ &
 $0$ &$2.782047$ &
 $0$ &$0.005951$ \\

& $0.8$ & $1.486422$ & $5.478970$ &
$0$ & $3.195397$ &
$0$ & $0.008972$ \\

 & $1.2$ & $0$ & $7.748070$ &
 $0.369579$ & $4.152256$ &
 $1.179972$ & $0.020209$ \\

 & $1.4$ & $0$ & $9.115696$ &
 $0$ & $4.692984$ &
 $0$ & $0.030159$\\
 \cline{1-8}

 \multirow{ 8}{*}[15pt]{ $0.4$  }&$0.2$ & $1.293017 $ & $3.642593$ & $2.894028$ & $1.310825$ &
 $0$ & $0.207280$ \\

 &$0.4$ & $2.113090$ & $4.287444$ &
 $0$ & $1.576915$ &
 $0$ & $0.259081$ \\

 &$0.6$ & $3.034171$ & $5.021366$ &
 $0$ & $1.888542$ &
 $0$ & $0.322611$ \\

& $0.8$ & $4.062510$ &  $5.850101$ & 
$0$ & $2.251073$ & 
$0$ & $0.400121$\\

 & $1.2$ & $0$ & $7.801728$ &
 $4.408508$ & $3.148673$ &
 $0$ & $0.607501$ \\

 & $1.4$ & $0$ & $8.916201$ &
 $3.804797$ & $3.691120$ &
 $0$ & $0.743157$\\
  \cline{1-8}

 \multirow{ 8}{*}[15pt]{ $0.5$ }&$0.2$ & $3.729161$ & $1.595773$ &
 $0$ & $0.003971$ &
 $0$ &$0.194751$ \\

 & $0.4$& $5.300419$ & $2.103792$ &
 $0$ & $0.006772$ & 
 $0$ & $0.244141$ \\

 &$0.6$& $0$ & $2.757935$ &
 $0$ & $0.011526$ &
 $0$ & $0.304907$\\

& $0.8$ & $0$ & $3.593937$ &
$0$ & $0.019565$ & 
$0$ &$0.379285$ \\

 & $1.2$ & $8.770850$ & $5.982027$ &
 $0$ & $0.055766$ & 
 $0$ & $0.579286$ \\

 & $1.4$ & $1.928976$ & $7.627746$ &
 $0$ & $0.093447$ & 
 $0$ & $0.710742$ \\
 \cline{1-8}

  \multirow{ 8}{*}[15pt]{ $0.6$ }&$0.2$ & $0$ & $1.862337$ &
 $0$ & $0.003360$ &
 $0$ &$0.055628$ \\

 & $0.4$& $0$ & $2.413613$ &
 $0$ & $0.005783$ & 
 $0$ & $0.073870$ \\

 &$0.6$& $0$ & $3.110758$ &
 $0$ & $0.009933$ &
 $0$ & $0.097748$\\

& $0.8$ & $0$ & $3.985814$ &
$0$ & $0.017020$ & 
$0$ &$0.128855$ \\

 & $1.2$ & $2.629053$ & $6.416391$ &
 $0$ & $0.049452$ & 
 $0$ & $0.221097$ \\

 & $1.4$ & $0$ & $8.047912$ &
 $0$ & $0.083682$ & 
 $0$ & $0.287556$ \\
 \cline{1-8}
 
\end{tabular}}
\end{center}
\end{table}

\begin{table}[!h] 
\begin{center}
\caption { PRI values of $\delta_{ST}$ and $\delta_{BZ}$ for generated samples of sizes $n=10,15,20,30$ with parameters $\sigma=1,u=0.1,0.2,0.3,0.4,0.5,0.6$ and $q=0.2,0.4,0.6,0.8,1.2,1.4$ } \label{table2}
{\begin{tabular}{ >{\centering\arraybackslash}p{0.8cm} >{\centering\arraybackslash}p{0.8cm}>{\centering\arraybackslash}p{1.4cm}>{\centering\arraybackslash}p{1.5cm}>{\centering\arraybackslash}p{0.8cm}>{\centering\arraybackslash}p{1.5cm}>{\centering\arraybackslash}p{0.8cm}>{\centering\arraybackslash}p{1.5cm}>
{\centering\arraybackslash}p{0.8cm}>
{\centering\arraybackslash}p{1.5cm}}
\hline
\multicolumn{2}{c}{n} &\multicolumn{2}{c}{10}&\multicolumn{2}{c}{15}&\multicolumn{2}{c}{20}&\multicolumn{2}{c}{30}\\
\hline
  $u$ & $q$ & $\delta_{ST}$ &$\delta_{BZ}$ & $\delta_{ST}$ &$\delta_{BZ}$  & $\delta_{ST}$ &$\delta_{BZ}$ & $\delta_{ST}$  & $\delta_{BZ}$ \\ 

\cline{1-10}
 \multirow{ 8}{*}[15pt]{ $0.1$} & $0.2$&$ 0$ & $1.432167$ &
 $0$  &$0.043261$ 
 &$0$  &$0.234838$ &
 $0$ & $0.051905$\\

 &$0.4$ & $0$  & $1.566019$ &
 $0$ & $0.050131$ &
 $0$ & $0.252638$  &
 $0$   &   $0.055340$\\

 &$0.6$ & $0$ & $1.707100$ &
 $0$ & $0.057962$ &
 $0$ & $0.271303$ &
$0$  &    $0.058930$\\

&$0.8$  & $0$ & $1.854917$ &
$0$ & $0.066861$ 
& $0$ & $0.290815$ &
 $0$   &   $0.062674$\\

 &$1.2$  & $0$ & $2.167615$ &
 $0$ & $0.088317$ &
 $0$  & $0.332256$ &
 $0$   &   $0.070618$\\

 &$1.4$  & $0.731157$ & $2.330193$  &
 $0$ & $0.101109$ &
 $0$ & $0.354093$ &
 $0$   &   $0.074811$\\
 \cline{1-10}

 
 \multirow{ 8}{*}[15pt]{ $0.2$ }&$0.2$ &$0$ & $0.381504$ & 
 $0$ & $0.005446$ & 
 $0$ & $0.058143$ &
 $0$    &  $0.001554$\\

 & $0.4$& $0$ & $0.442741$ &
 $0$ & $0.006633$ & 
 $0$  & $0.064302$ &
 $0$   &  $0.001730$\\

 &$0.6$ & $0$ & $0.512169$ &
 $0$ & $0.008063$ & 
 $0$ &$0.070986$ &
 $0$   &   $0.001923$\\

& $0.8$ & $0$ & $0.590511$ &
$0$ & $0.009781$ &
$0$ &$0.078223$ &
$0$   &    $0.002136$\\

 & $1.2$ & $0$ & $0.776622$ & 
 $0$ & $0.014294$ &
 $0$ & $0.094449$ &
 $0$   &   $0.002625$\\

 &$1.4$  & $0$  & $0.885490$ &
 $0$ & $0.017218$ &
 $0$ & $0.103478$ &
$0$    &    $0.002905$\\
 \cline{1-10}

 \multirow{ 8}{*}[15pt]{ $0.3$  }&$0.2$   & $0$ & $0.351306$ &
 $0$ & $0.004089$ & 
 $0$ & $0.000264$ &
 $0$   &   $0.000405$\\

 &$0.4$ & $0$ &$0.409050$ &
 $0$ & $0.005014$ & 
 $0$ & $0.000322$ &
 $0$   &   $0.000405$\\

 &$0.6$ & $0$ &$0.474765$ &
 $0$ &$0.006136$ &
 $0$ &$0.000392$ &
$0$    &    $0.000517$\\

& $0.8$ & $0$ & $0.549201$ &
$0$ & $0.007493$ &
$0$ & $0.000475$ &
 $0$   &   $0.000584$\\

 & $1.2$ & $0$ & $0.727082$ &
 $0$ & $0.011102$ &
 $0$ & $0.000698$ &
 $0$   &   $0.000741$\\

 & $1.4$ & $0$ & $0.831744$ &
 $0$ & $0.013466$ &
 $0$ & $0.000843$&
 $0$   &  $0.000833$\\
 \cline{1-10}

 \multirow{ 8}{*}[15pt]{ $0.4$  }&$0.2$ & $0$ & $0.111910$ &
 $0$ & $0.000116$ &
 $0$ & $0.000482$ &
 $0$   &  $0.000051$\\

 &$0.4$ & $0$ & $0.136201$ &
 $0$ & $0.000155$ &
 $0$ & $0.000581$ &
 $0$   &   $0.000061$\\

 &$0.6$ & $0$ & $0.165232$ &
 $0$ & $0.000206$ &
 $0$ & $0.000699$ &
 $0$   &   $0.000072$\\

& $0.8$ & $0$ &  $0.199788$ & 
$0$ & $0.000273$ & 
$0$ & $0.000840$&
 $0$  &   $0.000086$\\

 & $1.2$ & $0$ & $0.288980$ &
 $0$ & $0.000477$ &
 $0$ & $0.001207$ &
$0$   &   $0.000120$\\

 & $1.4$ & $0$ & $0.345529$ &
 $0$ & $0.000629$ &
 $0$ & $0.001443$&
 $0$   &   $0.000142$\\
  \cline{1-10}

 \multirow{ 8}{*}[15pt]{ $0.5$ }&$0.2$ & $0$ & $0.000322$ &
 $0$ & $0.010444$ &
 $0$ &$0.000626$ &
 $0$   &   $0$\\

 & $0.4$& $0$ & $0.000480$ &
 $0$ & $0.012527$ & 
 $0$ & $0.000751$ &
 $0$  &   $0$\\

 &$0.6$& $0$ & $0.000714$ &
 $0$ & $0.014994$ &
 $0$ & $0.000905$&
 $0$   &   $0$\\

& $0.8$ & $0$ & $0.001060$ &
$0$ & $0.017907$ & 
$0$ &$0.001077$ &
 $0$   &   $0$\\

 & $1.2$ & $0$ & $0.002326$ &
 $0$ & $0.025364$ & 
 $0$ & $0.001533$ &
 $0$ &   $0.000001$\\

 & $1.4$ & $0$ & $0.003433$ &
 $0$ & $0.030074$ & 
 $0$ & $0.001825$ &
 $0$   &   $0.000001$\\
 \cline{1-10}

 \multirow{ 8}{*}[15pt]{ $0.6$ }&$0.2$ & $0$ & $0.009820$ &
 $0$ & $0.000147$ &
 $0$ &$0.000128$ &
 $0$   &   $0$\\

 & $0.4$& $0$ & $0.013044$ &
 $0$ & $0.000195$ & 
 $0$ & $0.000157$ &
 $0$  &   $0$\\

 &$0.6$& $0$ & $0.017279$ &
 $0$ & $0.000257$ &
 $0$ & $0.000194$&
 $0$   &   $0$\\

& $0.8$ & $0$ & $0.022825$ &
$0$ & $0.000339$ & 
$0$ &$0.000238$ &
 $0$   &   $0$\\

 & $1.2$ & $0$ & $0.039449$ &
 $0$ & $0.000586$ & 
 $0$ & $0.000359$ &
 $0$ &   $0$\\

 & $1.4$ & $0$ & $0.051586$ &
 $0$ & $0.000769$ & 
 $0$ & $0.000439$ &
 $0$   &   $0$\\
 \cline{1-10}

\end{tabular}}
\end{center}
\end{table}

\end{document}